\documentclass[12pt]{article}
\pdfoutput=1
\usepackage{amsmath,amssymb,amsthm,amsfonts,mathrsfs}
\usepackage{wasysym,latexsym,times,lineno,color}
\usepackage{geometry}
\usepackage{graphicx,graphics}
\usepackage{epsfig}
\usepackage{subfigure}
\usepackage{epstopdf}
\usepackage{cite}
\usepackage{hyperref}
\usepackage{cleveref}

\newtheorem{theorem}{Theorem}[section]
\newtheorem{corollary}{Corollary}[section]

\newtheorem{lemma}[theorem]{Lemma}
\newtheorem{proposition}{Proposition}[section]

\theoremstyle{definition}

\begin{document}
\title{Dynamics of A Single Population Model with Memory Effect and Spatial Heterogeneity
}
\author{ Yujia Wang\textsuperscript{1},\;  Chuncheng Wang\textsuperscript{2},\; Dejun Fan\textsuperscript{3}
\\
{\small \textsuperscript{1,2} Department of Mathematics, Harbin Institute of Technology,\hfill{\ }}\\
\ \ {\small Harbin, Heilongjiang, 150001,  China.  \hfill{\ }}\\
{\small \textsuperscript{3} Department of Mathematics, Harbin Institute of Technology (Weihai), \hfill{\ }}\\
\ \ {\small Weihai, Shandong, 264209, China. \hfill {\ }}\\
}

\date{}

\maketitle
\begin{abstract}
In this paper, a single population model with memory effect and the heterogeneity of the environment, equipped with the Neumann boundary, is considered. The global existence of a spatial nonhomogeneous steady state is proved by the method of upper and lower solutions, which is asymptotically stable for relatively small memorized diffusion. However, after the memorized diffusion rate exceeding a critical value,
spatial inhomogeneous periodic solution can be generated through Hopf bifurcation, if the integral of intrinsic growth rate over the domain is negative. Such phenomenon will never happen, if only memorized diffusion or spatially heterogeneity is presented, and therefore must be induced by their joint effects. This indicates that the memorized diffusion will bring about spatial-temporal patterns in the overall hostile environment. When the integral of intrinsic growth rate over the domain is positive, it turns out that the steady state is still asymptotically stable. Finally, the possible dynamics of the model is also discussed, if the boundary condition is replaced by Dirichlet condition.
\end{abstract}

\section{Introduction}
Investigating the population dynamics has long been and will continue to be one of the dominant themes in ecology. In the process of mathematical modeling on this topic, the diffusion process and spatial heterogeneity have been often considered in studying the population dynamics \cite{Lam2014,diffusion,Lou2006,He1,m2}. The most simple model among these may read
\begin{equation}\label{proto}
\frac{\partial{u}}{\partial{t}}=d_1\Delta u+u(m(x)-u),\quad\quad x\in\Omega, t>0,
\end{equation}
where $d_1$ is the Fickian diffusion coefficient, ${\Omega}\subset\mathbb{R}^n$ is bounded and $m(x)$ describes the rate at which the population would grow or decline at the location $x$. However, it is suggested in \cite{ZongShu} that the memory effect should be incorporated in the diffusion process, especially for the highly developed animals. For this reason, we shall study the following model
\begin{equation}
\begin{cases}
\frac{\partial{u}}{\partial{t}}=d_1\Delta u+d_2\nabla\cdot(u\nabla u_r)+u(m(x)-u),\quad &x\in\Omega, t>0,\\
\partial_{\nu}u=0,\quad &x\in\partial\Omega, t>0,
\end{cases}
\label{m1}
\end{equation}
Here, the effect of memory on the spatial movement is characterized by the term $d_2\nabla\cdot(u\nabla u_r)$, and its derivation can be found in \cite{chuncheng}. The coefficient $d_2\in\mathbb{R}$ is the memory-based diffusion coefficient and $u_r=u(x,t-r)$ with $r>0$ representing the averaged memory period.

For \eqref{m1} with $d_2=0$, if $\overline{m}:=\frac{1}{|\Omega|}\int_\Omega m\mathrm{d}x<0$, it is well known that there is a non-homogeneous steady state $\bar{u}(x)$ for $0<d_1<\lambda_1(m)$, which is globally asymptotically stable, meaning that the population can survive only for small diffusion rate. When $\overline{m}>0$, $\bar{u}(x)$ exists for any $d_1>0$ and is also globally stable, see \cite{CCBOOK}.  Moreover, the detailed profiles of $\bar{u}(x)$ for both small and large diffusion coefficients are characterized in \cite{LouY}, showing that, $\bar{u}(x)$ tends to $m_+(x):=\max\{m(x),0\}$ ($\overline{m}$ resp.) when $d_1\rightarrow0$ ($d_1\rightarrow+\infty$ resp.). In addition, the total amount of population for the steady state $\bar{u}(x)$ is greater than $|\Omega|\bar{m}$ for any $d_1>0$ in this case.
Similar results can be also obtained if the reaction term in \eqref{m1} is replaced by $r(x)u(1-u/K(x))$, under some further conditions, see \cite{DeAngelis, He2019}.  When the maturation delay is considered, the existence of periodic solution is studied via Hopf bifurcation analysis in \cite{Song}. In \cite{controlC,controlL,Ding}, the authors focus on the optimal choice of $m(x)$ in \eqref{m1} with $d_2=0$ such that the steady state $\bar{u}(x)$ has the optimal spatial arrangement. For more reaction diffusion models with spatial heterogeneity, we refer the readers to  \cite{Ni2011,Lou2008,Lou-pre,Du2,Lam,He} and references therein.

For \eqref{m1} with $d_2\neq0$ and $m(x)=\text{constant}$, it has been shown in \cite{chuncheng} that the local stability of constant steady state is completely determined by the ratio of $d_1$ and $d_2$, but independent of memory delay $r$. However, if a maturation delay is considered in the reaction term, it turns out that the interaction of these two delays could induce stable spatially inhomogeneous periodic solutions through Hopf bifurcation, see \cite{Cdelay}. In \cite{nonlocal}, more rich dynamics are detected through higher codimension bifurcation analysis, such as Turing-Hopf and double Hopf bifurcations, when the nonlocal effect is considered in the reaction term.

The purpose of this paper is to investigate the model \eqref{m1}, particularly on the impact of memorized diffusion on dynamics of \eqref{proto}, under Neumann boundary condition. For convenience, by letting $t=\frac{\tilde{t}}{d_1}$, $\lambda=\frac{1}{d_1}$, $D=\frac{d_2}{d_1}$ and $\tau=rd_1$, the model \eqref{m1} can be transformed into:
\begin{equation}
\begin{cases}
\frac{\partial{u}}{\partial{t}}=\Delta u+D\nabla\cdot(u\nabla u_{\tau})+\lambda u(m(x)-u),\quad &x\in\Omega, t>0,\\
\partial_{\nu}u=0,\quad &x\in\partial\Omega, t>0.
\end{cases}
\label{model}
\end{equation}
We study the dynamics of \eqref{model} for the following two cases:
\begin{itemize}
  \item[\bf{(A1)}] $m(x)\in C^{\alpha}(\overline{\Omega})\left(\alpha\in(0,1)\right)$ is positive on a subset in $\Omega$, satisfying $\int_{\Omega}m(x)\mathrm{d}x<0$;
  \item[\bf{(A2)}] $m(x)\in C^{\alpha}(\overline{\Omega})$, $m(x)\not\equiv$ constant and
    $\int_{\Omega}m(x)\mathrm{d}x\geqslant0$.
\end{itemize}

For \eqref{model}, by studying the locally stability of the steady states,
we found that the positive steady states are always locally asymptotically stable for relatively small diffusion rate $D$ for both cases of $\bf{(A1)}$ and $\bf{(A2)}$. This means that small memorized diffusion rate $d_2$ does not affect local stability of the positive steady state, which can be expected from the results in \cite{CCBOOK}. However, when $D$ exceeds a critical value $\bar{D}$, the scenario will be different, and new dynamics can be induced by the effect of memory. Specifically, in the case of $\bf{(A1)}$, it will be shown that the increment of memory delay $\tau$ will cause the occurrence of Hopf bifurcation for \eqref{model}, leading to the existence of spatially inhomogeneous periodic solution, as long as $D>\bar{D}$.
We emphasize that such phenomenon must be induced by the combination of memory delay $\tau$ and $m(x)$, since it was already known that only one of these two factors will not drive \eqref{model} to generate stable spatial-temporal patterns, see \cite{chuncheng,CCBOOK}. For case $\bf{(A2)}$, we show that the stability of positive steady state only depends on the diffusion coefficient $D$, but is independent of $\tau$, which coincides with the stability results in \cite{chuncheng} where $m(x)$ is constant.

As in the reaction diffusion equation with Dirichlet boundary condition, the fact of spatial heterogeneity in the model \eqref{model} will also make the steady state to be inhomogeneous. For this reason, the corresponding characteristic equation is usually a complicated elliptic problem. In \cite{huang}, a method based on implicit function theorem is proposed to detect the eigenvalues with zero real parts for such problem, which has been widely used in various models, see\cite{shan2,shan1,shan3,guo1,guo2,shanNM}.
Recently, this method is also extended to equation with memorized diffusion and Dirichlet boundary condition in \cite{anqi}. However, a lot more prior estimations are required, since a non self-conjugate operator and time delay are involved in the characteristic equation. For studying the local stability and bifurcation of positive steady state of \eqref{model}, we mainly used the method in \cite{anqi}. But, the techniques are totally different for proving the prior estimation of the eigenvalues, in the case of ${\bf(A2)}$.

The paper is organized as follows. In Section 2, we present the main results on \eqref{model} for both cases $\bf{(A1)}$ and $\bf{(A2)}$. Numerical simulations are also provided in this section. Sections 3, 4 and 5 are devoted to the proofs of main results. In section 6, the dynamics of \eqref{model}, as well as numerical outcomes, are discussed, when the boundary condition of \eqref{model} is replaced by Dirichlet condition.

\section{Main Results}

Let $X=H^2({\Omega}), Y=L^2(\Omega)$ and $\emph{C}=C([-\tau,0],Y)$. For any space $Z$, we define the complexification of $Z$ by $Z_{\mathbb{C}}:=Z\oplus iZ=\{x_1+ix_2|x_1,x_2\in Z\}$.
For the complex-valued Hilbert space $Y_{\mathbb{C}}$, the inner product is $\langle u,v \rangle=\int_{\Omega}\overline{u}(x)v(x)\mathrm{d}x$.

Considering the eigenvalue problem
\begin{equation}
\begin{cases}
-\Delta u(x)=\lambda m(x)u(x),\quad & x\in\Omega,\\
\partial_{\nu}u(x)=0,\quad & x\in\partial{\Omega}.
\end{cases}
\label{eigenvalue prombel}
\end{equation}
From \cite{CCBOOK}, it has a unique positive principal eigenvalue $\lambda_*$ with a positive nonconstant eigenfunction $\phi(x)\in $ $H^2(\overline{\Omega})$ under $\bf{(A1)}$, and a zero eigenvalue $\lambda_*=0$ with positive constant eigenfunction $\phi$ under $(\bf{A2})$. Note that $\phi(x)\in C^{2+\alpha}(\overline{\Omega})$ can be deduced from embedding theorems \cite{adams2003sobolev} and regularity theory for elliptic equations \cite{Gilbarg2001}. Without loss of generally, we assume that $\|\phi\|_{Y}=1$.

\begin{theorem}\label{existence}
Assume that
$$
\text{$\bf(O)$}\quad D>-\frac{1}{\max\limits_{x\in\overline{\Omega}}\{m(x)\}}.
$$
Then,
\begin{itemize}
\item[$(1)$] For case $\bf(A1)$,
\eqref{model} has a positive equilibrium $u_{\lambda}$ for $\lambda>\lambda_*$;
\item[$(2)$] For case $\bf(A2)$,
\eqref{model} has a positive equilibrium $u_{\lambda}$ for $\lambda>0$;
\end{itemize}
Moreover, $u_\lambda$ is unique provided that $D>0$, and the total population size of species satisfies
\begin{equation}\label{inequality}
\int_{\Omega}u_{\lambda}\mathrm{d}x>\int_{\Omega}m(x)\mathrm{d}x.
\end{equation}
\end{theorem}

We remark that the assumption $\bf{(O)}$ guarantees the global existence of steady state for $\lambda>\lambda_*$ or $\lambda>0$, which may not be necessary for the local existence of steady state. For studying local dynamics of $u_\lambda$ in the case of $\bf(A1)$, denote
$$
r_1=\int_{\Omega}\lambda_*\phi(x)^3\mathrm{d}x>0 \quad \text{and}\quad r_2=D\int_{\Omega}\phi(x)\nabla\cdot(\phi(x)\nabla\phi(x))\mathrm{d}x.
$$
We also make the following assumptions in this case
\begin{equation*}\label{conditionA1}
\begin{split}
\text{$\bf{(H1)}$} \quad r_1-r_2>0, \quad\quad
  &\text{$\bf{(H2)}$}\quad |D|<D_*:=\frac{1}{\max\limits_{\lambda\in(\lambda_*,\lambda^*],\overline{\Omega}}\{u_{\lambda}\}},  \\
\text{$\bf{(H3)}$} \quad r_1+r_2>0, \quad\quad
  &\text{$\bf{(H4)}$} \quad r_1+r_2<0,\\
\end{split}
\end{equation*}
where $u_{\lambda}$ is the steady state of \eqref{model} for $\lambda\in(\lambda_*,\lambda^*]$ with $0<{\lambda}^*-\lambda_*\ll1$.

\begin{theorem}\label{thconclusion}
Assume $\bf(A1)$ holds.
\begin{itemize}
\item[$(1)$] The following statements are valid.
           \begin{itemize}
            \item[$(i)$] If $\bf(H1)$, $\bf(H2)$ and $\bf(H3)$ satisfied, then there exists $\lambda_*<\tilde{\lambda}^*<\lambda^*$ such that for any $\lambda\in(\lambda_*,\tilde{\lambda}^*]$ and $\tau\in[0,\infty)$, $u_{\lambda}$ is locally asymptotically stable.
            \item[$(ii)$] If $\bf(H1)$, $\bf(H2)$ and $\bf(H4)$ satisfied, then for any $\lambda\in(\lambda_*,\tilde{\lambda}^*]$, there exist a sequence $\{\tau_n\}_{n=0}^{\infty}$ such that $u_{\lambda}$ is locally asymptotically stable for $\tau\in[0,\tau_0)$, unstable for $\tau\in(\tau_0,\infty)$, and \eqref{model} undergoes Hopf bifurcation at $\tau=\tau_0$.
           \end{itemize}
\item[$(2)$] If the inequality of $\bf(H1)$ is reversed, then for $\lambda\in[\underline{\lambda},\lambda_*)$
           with $0<\lambda_*-\underline{\lambda}\ll1$, $u_{\lambda}$ is unstable and the characteristic equation of $u_{\lambda}$ has pure imaginary roots, but for this case, the bifurcated periodic solutions are always unstable.
\end{itemize}
\label{conclusion}
\end{theorem}

\begin{figure}
\centering
\begin{tabular}{ccc}
 \includegraphics[height=6.5cm,width=12cm]{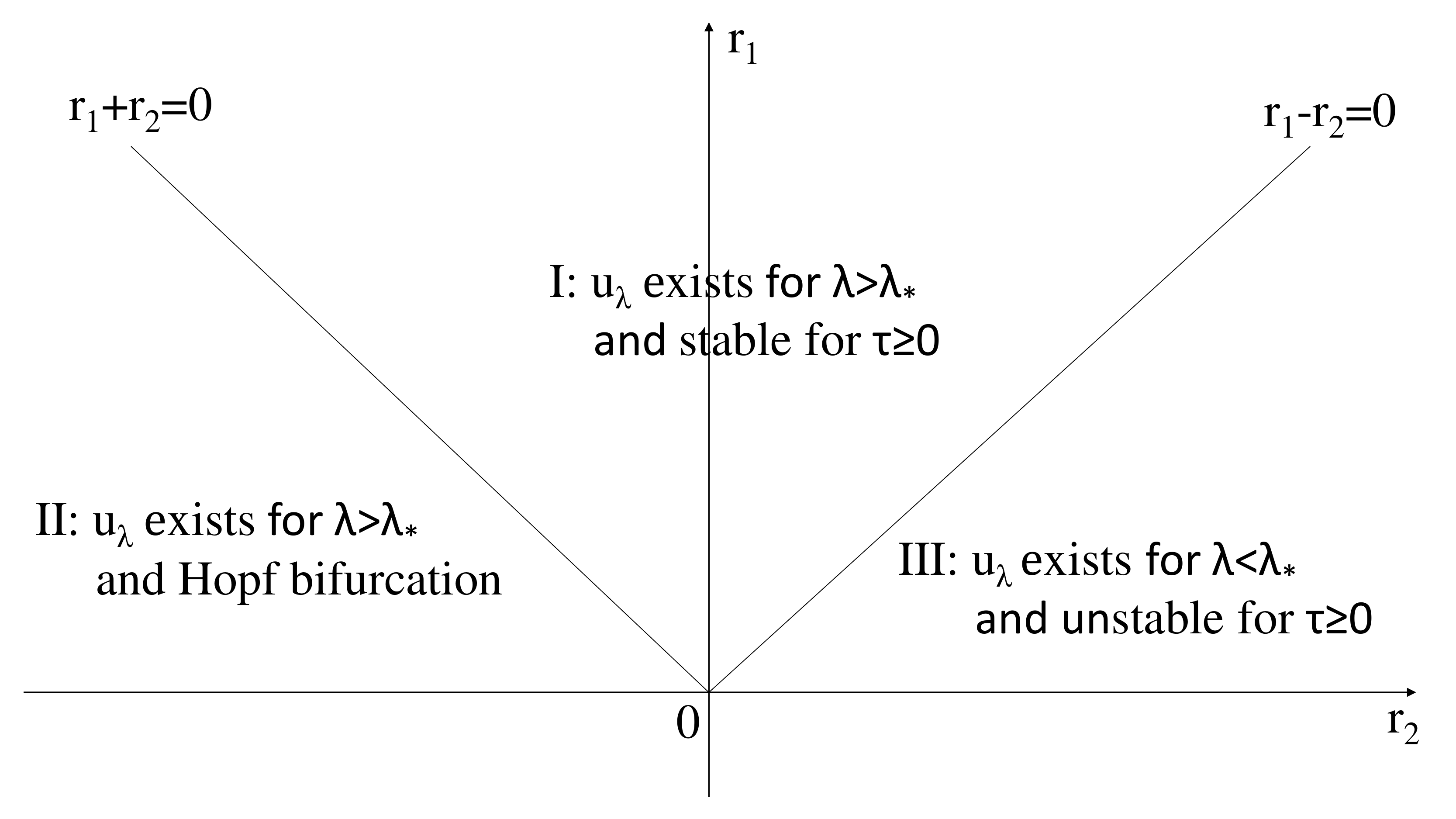}
\\
\end{tabular}
\caption{The local dynamics of \eqref{model} near $u_{\lambda}$. In region I, $u_{\lambda}$ is stable for $\tau\geqslant0$. In region II, the increment of delay can induce Hopf bifurcation at $u_{\lambda}$. In region III, $u_{\lambda}$ is always unstable for $\tau\geqslant0$.}
\label{picture}
\end{figure}

The results of Theorem \ref{thconclusion} are visualized in Figure \ref{picture}, where local dynamics of $u_\lambda$ are characterized in $(r_1,r_2)$ plane. For numerical test of Theorem \ref{thconclusion},
set $m(x)=-x^3+5$ for $x\in\overline{\Omega}=[0,\pi]$. Then, $\int_\Omega mdx=\frac{20\pi-\pi^4}{4}<0$, satisfying $\bf(A1)$. Moreover, $\lambda_*\approx0.0560, r_1\approx0.0755, \frac{r_2}{D}\approx-0.1100$. We can obtain the critical value $\overline{D}=0.6864$.
If we choose $\lambda=0.6>\lambda_*$ and $D=0.3$, then $r_1-r_2\approx0.1085>0$ and $r_1+r_2\approx0.0425>0$, which satisfy $\bf{(H1)}$ and $\bf{(H3)}$. From Theorem \ref{thconclusion} $(1)-(i)$, the steady state $u_{\lambda}$ is locally asymptotically stable for any $\tau\geq0$, see Figure \ref{NMle}-$(a)$. As $D$ is increased to $0.8$, we get $r_1-r_2\approx0.1635>0$ and $r_1+r_2\approx-0.0125<0$. It then follows from Theorem \ref{thconclusion} $(1)-(ii)$ that \eqref{model} will undergo Hopf bifurcation as $\tau$ raises, see Figure \ref{NMle}-$(b)$ and $(c)$.

\begin{figure}[htbp]
    \centering
    \subfigure[$D=0.3, \tau=50$]{
        \includegraphics[width=0.31\textwidth]{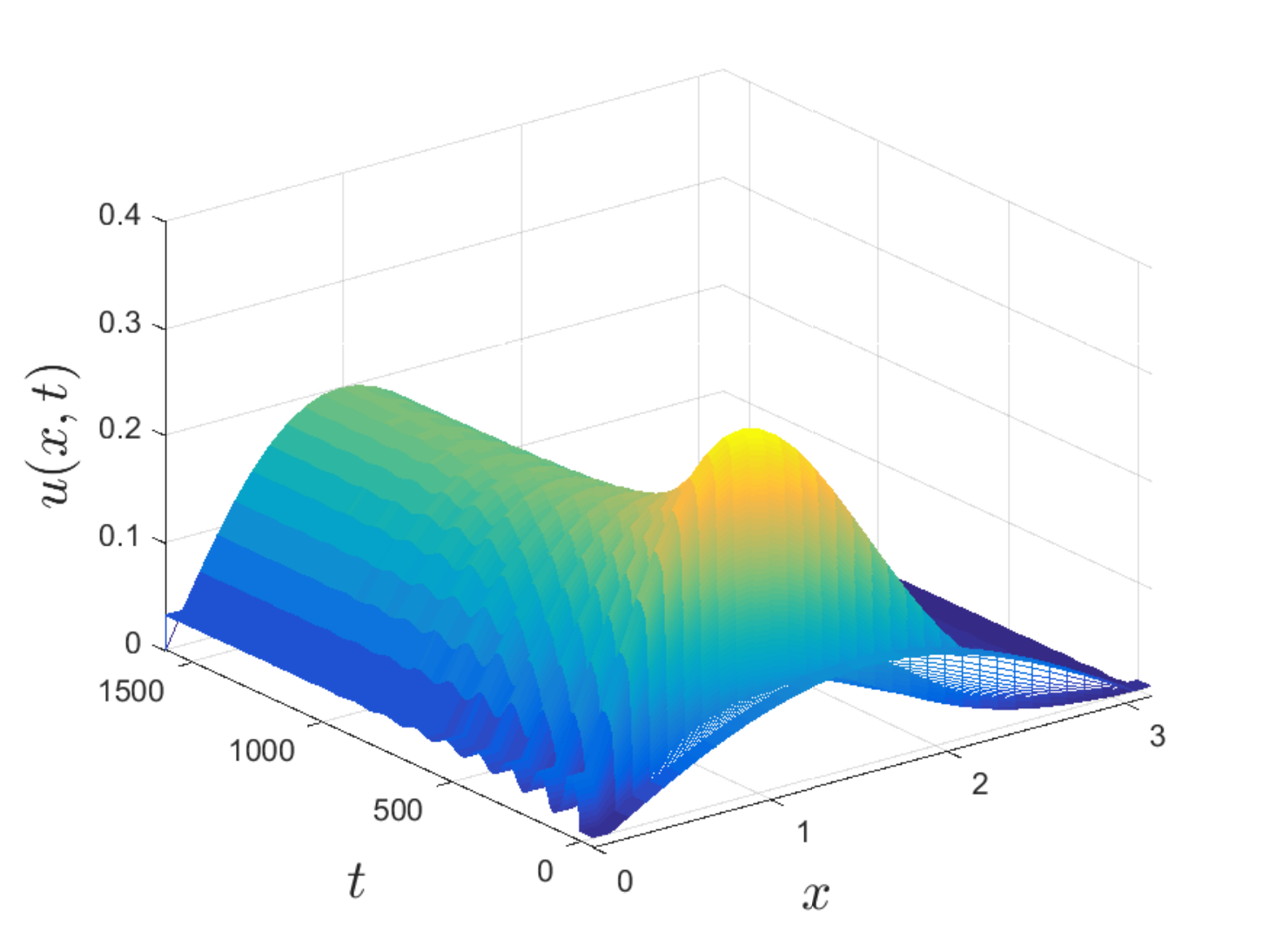}
    }
    \subfigure[$D=0.8, \tau=10$]{
        \includegraphics[width=0.31\textwidth]{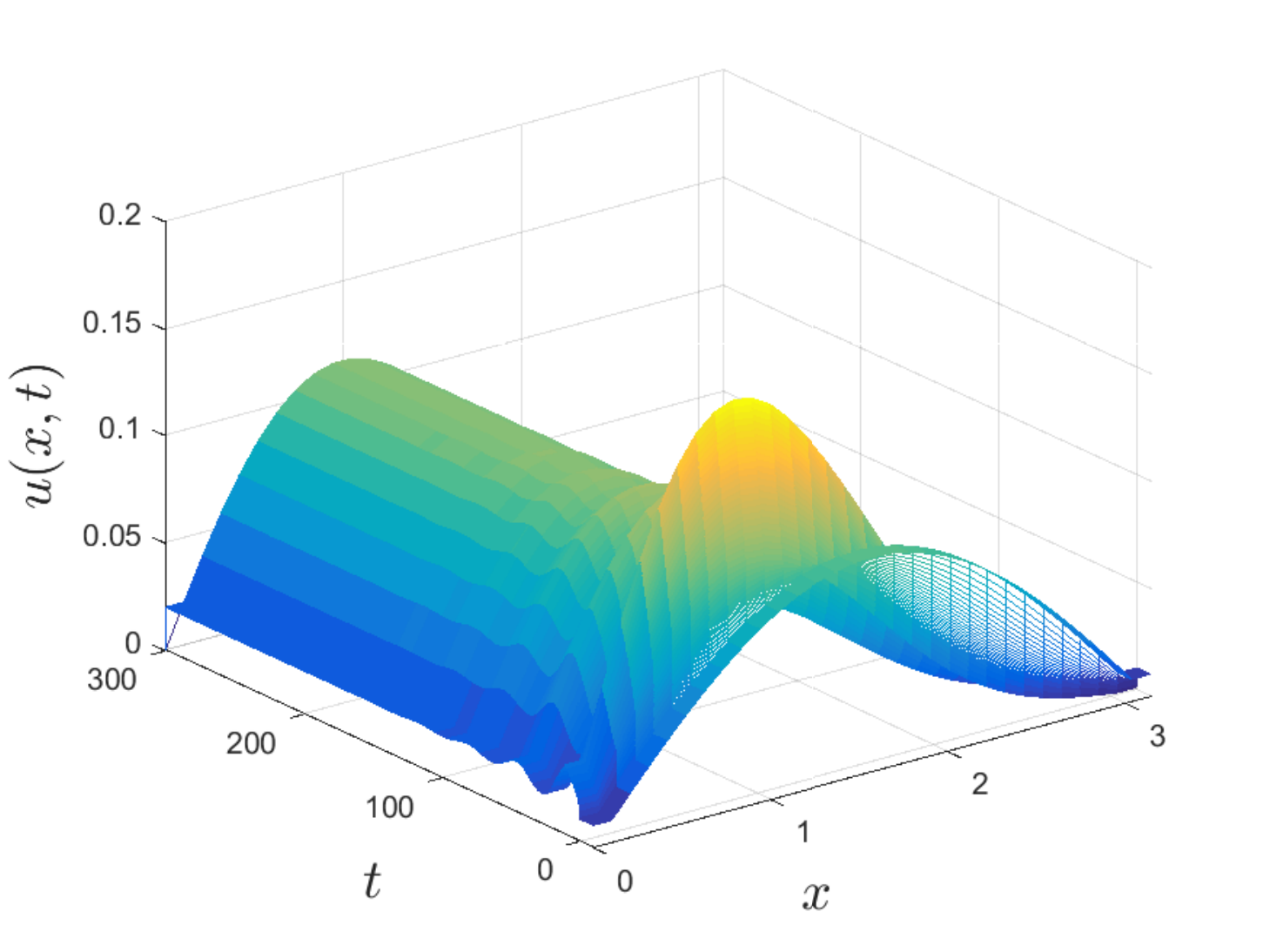}
    }
    \subfigure[$D=0.8, \tau=50$]{
        \includegraphics[width=0.31\textwidth]{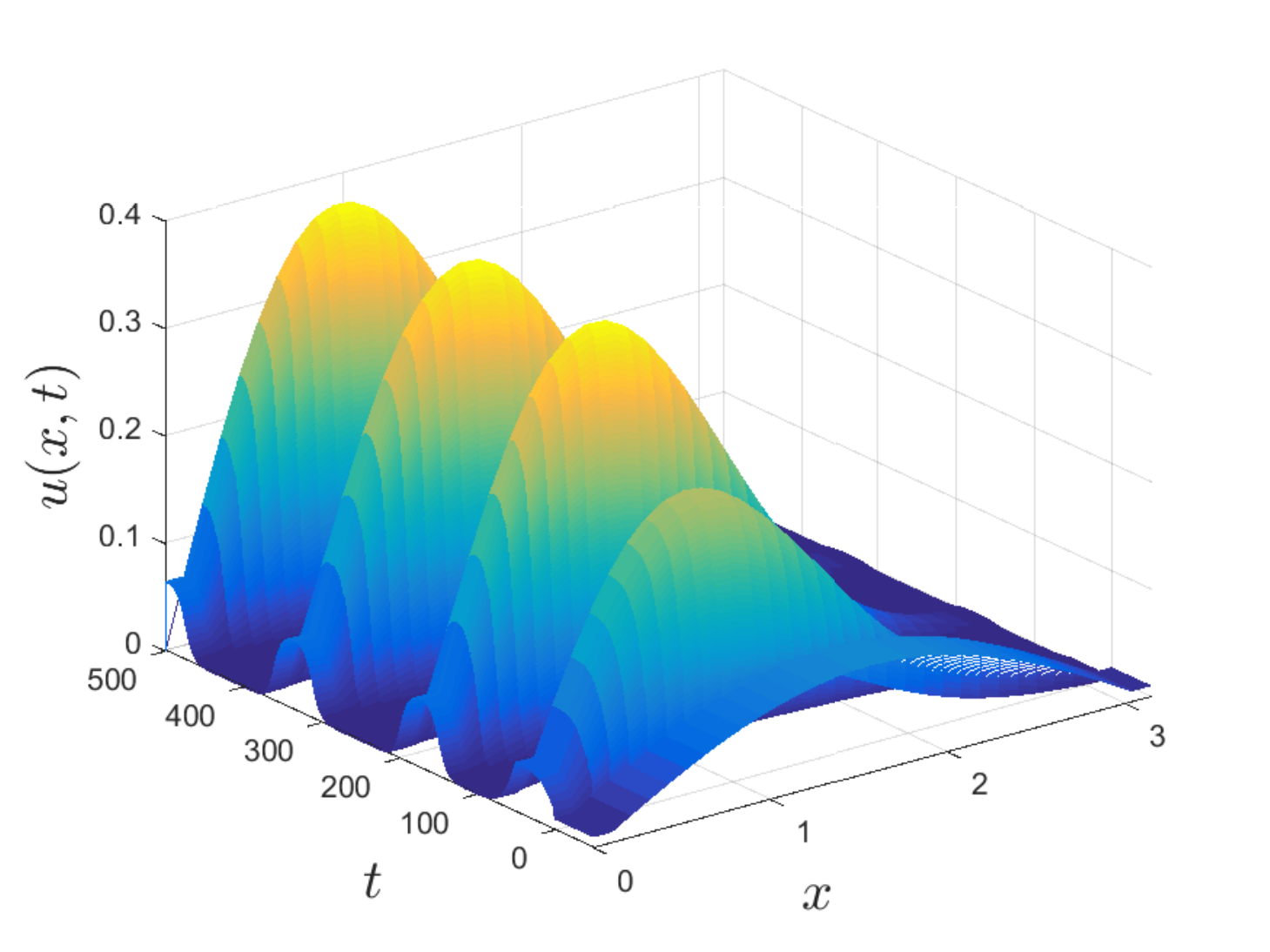}
    }
    \caption{Solutions of model (\ref{model}) for different choices of $(D,\tau)$. Here, $\Omega=[0,\pi]$, $m(x)=-x^3+5$, $\lambda=0.6$.}
\label{NMle}
\end{figure}

Now, for case $\bf(A2)$, we make the following assumptions:
\begin{equation*}\label{conditionA2}
\begin{split}
&\text{$\bf(P)$}\quad |D|<D^*:=\frac{1}{\max\limits_{\lambda\in[0,\lambda^*],\overline{\Omega}}\{u_{\lambda}\}}.
\end{split}
\end{equation*}

\begin{theorem}\label{conclusion2}
Under assumption $\bf(A2)$, if $\bf(P)$ satisfied, then for any $\lambda\in(0,\tilde{\lambda}^*]$ with $0<\tilde{\lambda}^*\ll1$, the positive steady state $u_{\lambda}$ of system \eqref{model} is locally asymptotically stable with $\tau\in[0,\infty)$.
\end{theorem}

We choose $m(x)=\sin x+1\geq0$ on $\overline{\Omega}=[0,\pi]$, and $m(x)=5\cos x+0.3$ which changes sign on $\Omega$ in numerical simulations. In either case, it is observed in Figure \ref{nuemann-min} that $u_\lambda(x)$ is locally asymptotically stable.

\begin{figure}[htbp]
	\centering
	\begin{minipage}{0.45\linewidth}
		\centering\subfigure[$m(x)=\sin x+1$]
{\includegraphics[scale=0.5]{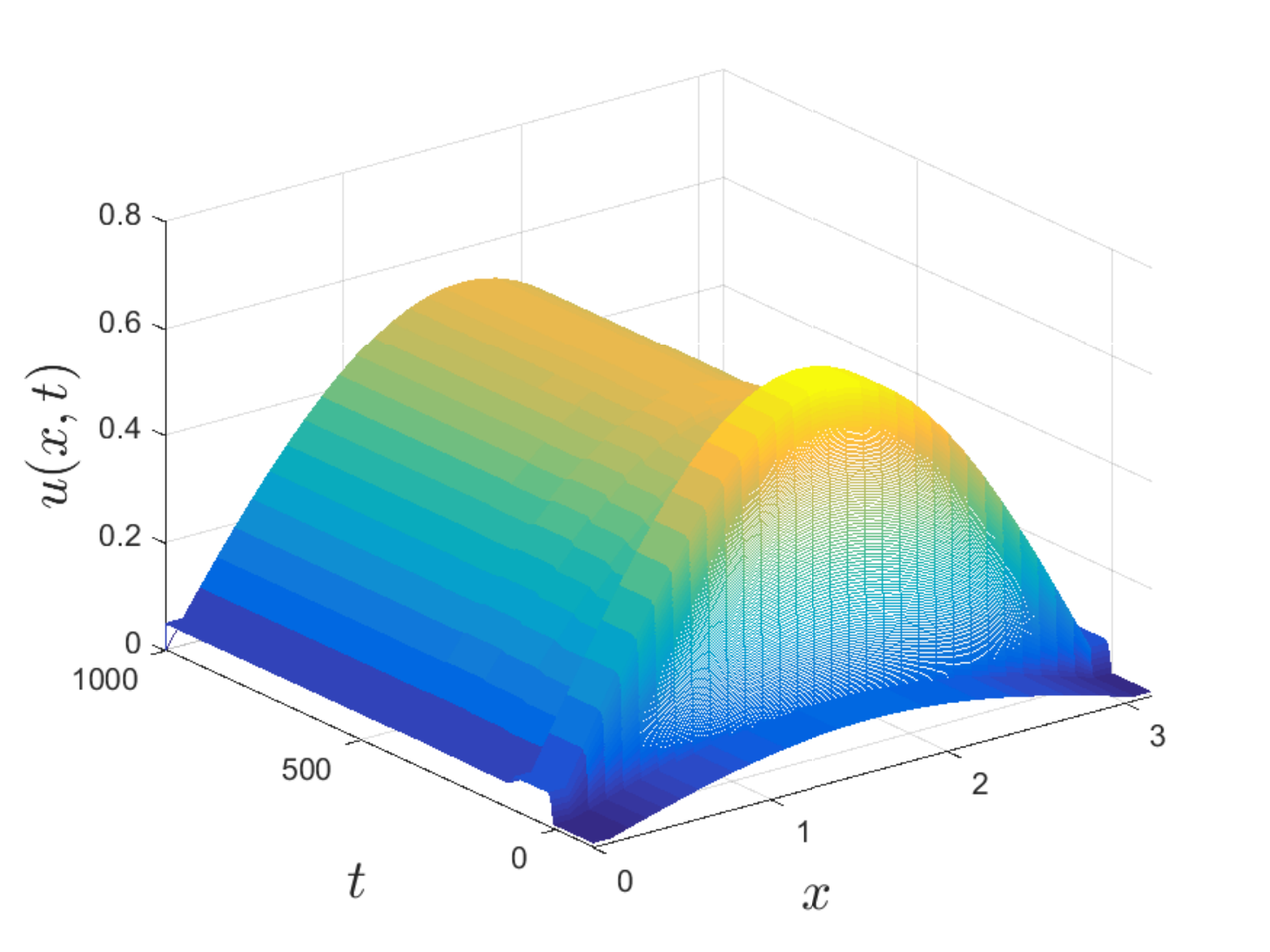}}
			\label{figA}
	\end{minipage}
	\begin{minipage}{0.45\linewidth}
		\centering\subfigure[$m(x)=5\cos x+0.2$]
{\includegraphics[scale=0.5]{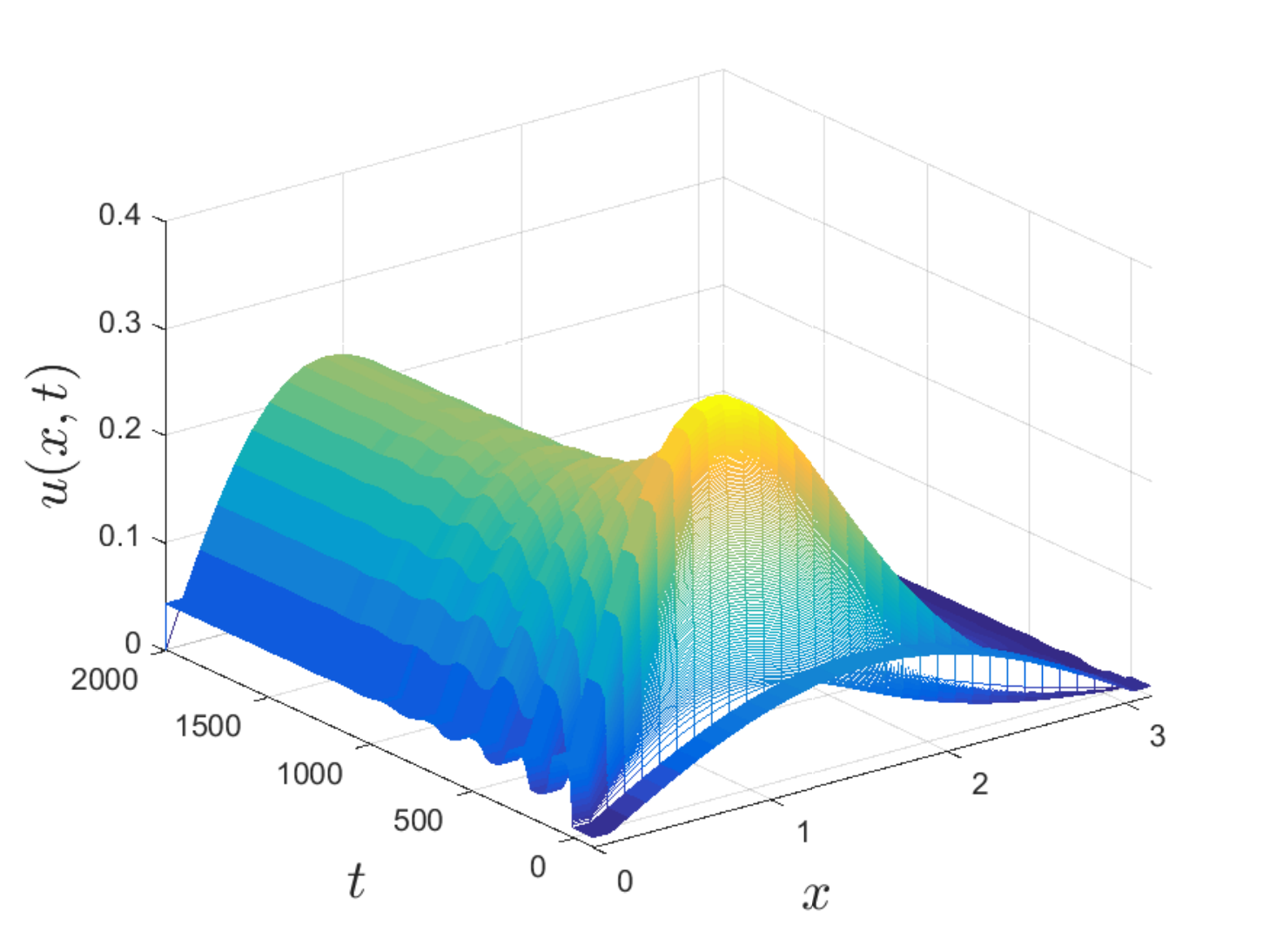}
			\label{figB}}
	\end{minipage}
\caption{The solution of \eqref{model} with different $m(x)$ satisfying ${\bf(A2)}$. Here, $\lambda=0.8$, $D=0.3$ and $\tau=100$.}
\label{nuemann-min}
\end{figure}

\section{The proof of Theorem \ref{existence}}\label{exist}

The existence of steady state in Theorem \ref{existence} will be shown by the method of upper and lower solution.
Let $\sigma_1$ is a principal eigenvalue (with positive eigenfunction $\varphi_1$) of
\begin{equation}
\begin{cases}
\Delta \varphi+\lambda m(x)\varphi=\sigma\varphi,\quad & x\in\Omega,\\
\partial_{\nu}\varphi=0,\quad & x\in\partial{\Omega}.
\end{cases}
\label{sigma}
\end{equation}
Then, under assumption $\bf(A1)$ ($\bf(A2)$ resp.), we have $\sigma_1>0$, when $\lambda>\lambda_*$ ($\lambda>0$, resp.).
Let $\underline{u}=\varepsilon\varphi_1$. It follows from that
\begin{equation}
\begin{split}
&\Delta(\varepsilon\varphi_1)+D(\varepsilon\varphi_1)\Delta(\varepsilon\varphi_1)+D\nabla(\varepsilon\varphi_1)\cdot\nabla(\varepsilon\varphi_1)+\lambda(\varepsilon\varphi_1)(m(x)-\varepsilon\varphi_1)\\
=&\varepsilon\sigma_1\varphi_1+\varepsilon^2\left(D\Delta\varphi_1-\lambda\varphi_1^2+D(\nabla\varphi_1)^2\right)\\
\end{split}
\end{equation}
which is positive, for sufficiently small $\varepsilon>0$.
This means $\underline{u}$ is a subsolution for \eqref{model}. On the other hand, it can be verified that $\overline{u}=\max\limits_{x\in\overline{\Omega}}\{m(x)\}:=K$ is a supersolution of \eqref{model}, and $\underline{u}\leqslant\overline{u}$ for sufficiently small $\varepsilon>0$.
By $\bf(O)$, we know $\bf(H3)$ in \cite{Pao} is satisfied. It then follows from Theorem 3.1 in \cite{Pao} that \eqref{model} admints a solution $u_{\lambda}$ such that $\underline{u}\leqslant u_{\lambda}\leqslant\overline{u}$.

Now, we are about to prove the uniqueness of $u_\lambda$ for $D>0$. Assume that $u_{\lambda}, \tilde{u}_{\lambda}$ are positive steady states of \eqref{model} such that $u_{\lambda}\not\equiv \tilde{u}_{\lambda}$ and $u_{\lambda}<\tilde{u}_{\lambda}$ somewhere on $\Omega$. Since $u_{\lambda}$ is a positive solution to
\begin{equation}\label{u*}
\begin{cases}
\nabla\cdot\left((1+Du_{\lambda})\nabla\varphi\right)+
\lambda\left(m(x)-u_{\lambda}\right)\varphi=\sigma\varphi,\quad & x\in\Omega,\\
\partial_{\nu}\varphi=0,\quad & x\in\partial{\Omega},
\end{cases}
\end{equation}
for $\sigma=0$, we know that $\sigma_1=0$ is the principal eigenvalue for \eqref{u*}. Similarly, $\tilde{u}_{\lambda}>0$ solves
 \begin{equation}\label{u**}
\begin{cases}
\nabla\cdot\left((1+D\tilde{u}_{\lambda})\nabla\tilde\varphi\right)+\lambda\left(m(x)-\tilde{u}_{\lambda}\right)\tilde\varphi=\tilde\sigma\tilde\varphi,\quad & x\in\Omega,\\
\partial_{\nu}\tilde\varphi=0,\quad & x\in\partial{\Omega},
\end{cases}
\end{equation}
with $\tilde\sigma=0$. So $\tilde\sigma_1=0$ is the principle eigenvalue for \eqref{u**}. Obviously, $m(x)-u_{\lambda}>m(x)-\tilde{u}_{\lambda}$. If $D>0$, then $1+Du_{\lambda}<1+D\tilde{u}_{\lambda}$. From Corollary 2.2 in \cite{CCBOOK}, the principal eigenvalue of \eqref{u*} must be greater than the one of \eqref{u**}, which leads to a contradiction.

In order to prove \eqref{inequality} in the case of $\bf(A2)$ (note that \eqref{inequality} is always true under $\bf(A1)$), we first show that $u_\lambda>0$ on $\bar{\Omega}$. Since $u_\lambda$ is bounded and solves
\begin{equation}
\begin{cases}
\Delta u_{\lambda}+D\nabla\cdot(u_{\lambda}\nabla u_{\lambda})+\lambda u_{\lambda}(m(x)-u_{\lambda})=0,\quad &x\in\Omega,\\
\partial_{\nu}u=0,            &x\in\partial\Omega,
\end{cases}
\label{inte}
\end{equation}
we know that
$u_{\lambda}\in C^{2+\gamma}(\Omega)$, $0<\gamma<\frac{1}{2}$, by the embedding theorems and regularity theory for elliptic equation.
Using Harnack inequality \cite{Gilbarg2001}, we have $u_{\lambda}(x)>0$ for $x\in\Omega$.  If there exist $x_0\in\partial\Omega$ such that $u_{\lambda}(x_0)=0$, then for any $x\in\Omega$, $u_{\lambda}(x)>u_{\lambda}(x_0)$ and from the first equation of \eqref{inte},
we have $\partial_{\nu}u(x_0)<0$ by Hopf lemma \cite{wanghopf}, which contracts with $\partial_{\nu}u(x_0)=0$. Thus, $u_{\lambda}>0$ in $\overline{\Omega}$.

Multiplying both sides of the first equation of \eqref{inte} by $\frac{1}{u_{\lambda}}$ and then integrating over $\Omega$, we have
\begin{equation*}
\begin{split}
\lambda\int_{\Omega}u_{\lambda}\mathrm{d}x
    &=\lambda\int_{\Omega}m(x)\mathrm{d}x+\int_{\Omega}\frac{|\nabla u_{\lambda}|^2}{u_{\lambda}^2}\mathrm{d}x
        +D\int_{\Omega}\frac{{|\nabla u_{\lambda}|^2}}{u_{\lambda}}\mathrm{d}x  \\
    &\geqslant\lambda\int_{\Omega}m(x)\mathrm{d}x+\frac{1}{\max\limits_{\lambda\in(0,+\infty),\overline{\Omega}}\{u_{\lambda}\}}\int_{\Omega}
         \frac{|\nabla u_{\lambda}|^2}{u_{\lambda}}\mathrm{d}x+D\int_{\Omega}\frac{|\nabla u_{\lambda}|^2}{u_{\lambda}}\mathrm{d}x \\
    &=\lambda\int_{\Omega}m(x)\mathrm{d}x+\left[\frac{1}{\max\limits_{\lambda\in(0,+\infty),\overline{\Omega}}\{u_{\lambda}\}}
        +D\right]\int_{\Omega}\frac{|\nabla u_{\lambda}|^2}{u_{\lambda}}\mathrm{d}x.
\end{split}
\end{equation*}
From ${\bf(O)}$,
we obtain
$$
\int_{\Omega}u_{\lambda}(x)\mathrm{d}x>\int_{\Omega}m(x)\mathrm{d}x.
$$
This completes the proof of Theorem \ref{existence}.

\section{The proof of Theorem \ref{conclusion}}\label{A1}

The steady states of \eqref{model} are determined by the following problem,
\begin{equation}
\begin{cases}
\Delta u(x)+D\nabla\cdot\left(u(x)\nabla u(x)\right)+\lambda u(x)\left(m(x)-u(x)\right)=0, \quad &x\in\Omega,\\
\partial_{\nu}u(x)=0, \quad &x\in\partial{\Omega}.
\end{cases}
\label{boundary}
\end{equation}
Defining the nonlinear operator $T:X\times\mathbb{R}_+\rightarrow Y$ by
$$
T(u,\lambda)=\Delta u+D\nabla\cdot(u\nabla u)+\lambda u(m(x)-u).
$$
Notice that $D_uT(0,\lambda_*)=\Delta+\lambda_*m(x)$ is the self-conjugate 
Fredholm operator from $X\rightarrow Y$ with index zero. Thus,
$$
X=\mathrm{Ker}(\Delta+\lambda_*m(x))\oplus X_1,
$$
$$
Y=\mathrm{Ker}(\Delta+\lambda_*m(x))\oplus Y_1,
$$
where
$$
\mathrm{Ker}(\Delta+\lambda_*m(x))=\mathrm{span}\{\phi\},
$$
$$
X_1=\left\{y\in X:\int_{\Omega}\phi(x)y(x)dx=0\right\},
$$
$$
Y_1=\mathrm{Range}\left\{\Delta+\lambda_*m(x)\right\}=\left\{y\in Y:\int_{\Omega}\phi(x)y(x)\mathrm{d}x=0\right\}.
$$
Obviously, the operator $[\Delta+\lambda_*m(x)]|_{X_1}:X_1\rightarrow Y_1$ is invertible and has a bounded inverse.

\begin{proposition}\label{stst}
Assume that \text{\bf{(H1)}} holds. Then there exists the continuous differentiable mapping $\lambda\longmapsto(\xi_{\lambda},\alpha_{\lambda})$ from $[\lambda_*,\lambda^*]$ to $X_1\times\mathbb{R^+}$ with $0<\lambda^*-\lambda_*\ll1$, such that for any $\lambda\in[\lambda_*,\lambda^*]$, \eqref{boundary} has a steady state with the form of
$$
u_{\lambda}=\alpha_{\lambda}(\lambda-\lambda_*)[\phi+(\lambda-\lambda_*)\xi_{\lambda}].
$$
Moreover, for $\lambda=\lambda_*$,
$$
\alpha_{\lambda_*}=\frac{\int_{\Omega}m(x)\phi^2\mathrm{d}x}{r_1-r_2},
$$
and $\xi_{\lambda_*}\in X_1$ is the unique solution of the following equation
\begin{equation}
[\Delta+\lambda_*m(x)]\xi+\alpha_{\lambda_*}D\nabla\cdot(\phi\nabla\phi)+[m(x)-\lambda_*\alpha_{\lambda_*}\phi]\phi=0.
\label{xi}
\end{equation}
\end{proposition}

\begin{proof}
From \eqref{eigenvalue prombel}, we have
\begin{equation*}
\lambda_*\int_{\Omega}m(x)\phi^2(x)\mathrm{d}x
=\int_{\Omega}|\nabla\phi(x)|^2\mathrm{d}x>0,
\end{equation*}
which implies that $\alpha_{\lambda_*}$ is positive. By the definition of $\alpha_{\lambda_*}$, we know that
$$
\alpha_{\lambda_*}D\nabla\cdot(\phi\nabla\phi)+[m(x)-\lambda_*\alpha_{\lambda_*}\phi]\phi\in \mathrm{Range}\{\Delta+\lambda_*m(x)\}=Y_1.
$$
Therefore, \eqref{xi} has the unique solution $\xi_{\lambda_*}$.
Define the mapping $f:X_1\times\mathbb{R}\times\mathbb{R}_+\rightarrow Y$ by
$$
f(\xi,\alpha,\lambda)=[\Delta+\lambda_* m(x)]\xi+\alpha D\nabla \cdot[f_1(\xi,\lambda)\nabla f_1(\xi,\lambda)]+[m(x)-\lambda\alpha f_1(\xi,\lambda)]f_1(\xi,\lambda)
$$
where $f_1(\xi,\lambda)=\phi+(\lambda-\lambda_*)\xi$. By the definition of $\xi_{\lambda_*}$, we have $f(\xi_{\lambda_*},\alpha_{\lambda_*},\lambda_*)=0$
and
$$
D_{(\xi,\alpha)}f(\xi_{\lambda_*},\alpha_{\lambda_*},\lambda_*)(\epsilon,\kappa)=[\Delta+{\lambda_*}m(x)]\epsilon+D\nabla\cdot(\phi\nabla\phi)\kappa-\lambda_*\phi^2\kappa,
$$
where $(\epsilon,\kappa)\in X_1\times\mathbb{R}$.
By ${\bf(H1)}$, we have
$$
D\nabla\cdot(\phi\nabla\phi)-\lambda_*\phi^2\notin \mathrm{Range}\{\Delta+\lambda_*m(x)\}.
$$
Therefore, $D_{(\xi,\alpha)}f(\xi_{\lambda_*}, \alpha_{\lambda_*}, {\lambda_*})$ is bijection from $X_1\times\mathbb{R}\rightarrow Y$. From implicit function theorem, there exits a continuous differential mapping $\lambda\longmapsto(\xi_{\lambda},\alpha_{\lambda})$ from $[\lambda_*,\lambda^*]$ to $X_1\times\mathbb{R_+}$ with $\lambda^*>\lambda_*$, such that for any $\lambda\in[\lambda_*,\lambda^*]$, $f(\xi_{\lambda},\alpha_{\lambda},\lambda)=0$. Then
 $$
 u_{\lambda}=\alpha_{\lambda}(\lambda-\lambda_*)[\phi+(\lambda-\lambda_*)\xi_{\lambda}]
 $$
 is the solution of \eqref{boundary}.
\end{proof}

Thoughout this section, denote $$f_1(\xi,\lambda)=\phi+(\lambda-\lambda_*)\xi,\quad u_{\lambda}=(\lambda-\lambda_*)\alpha_{\lambda}f_1.$$
The linearized equation of \eqref{model} at $u_{\lambda}(x)$ is given by
 \begin{equation}
 \begin{cases}
 \frac{\partial{u}}{\partial{t}}=\Delta u+D\nabla\cdot(u_{\lambda}\nabla u_{\tau})+D\nabla\cdot(u\nabla
      u_{\lambda})+\lambda[m(x)-u_{\lambda}]u-\lambda u_{\lambda}u,\\
 \partial_{\nu}u=0.
 \end{cases}
 \label{linear}
 \end{equation}
For each $(\mu,\lambda,\tau)\in\mathbb{C}\times\mathbb{R}_{+}^2$, we introduce an linear operator $\Delta(\lambda,\mu,\tau)$ on $X_{\mathbb{C}}$ by
\begin{equation}
\Delta(\lambda,\mu,\tau)\psi=\Delta\psi+D\nabla\cdot(u_{\lambda}\nabla\psi)e^{-\mu\tau}+D\nabla\cdot(\psi\nabla u_{\lambda})+\lambda[m(x)-u_{\lambda}]\psi-\lambda u_{\lambda}\psi-\mu\psi=0.
\label{characteristic}
\end{equation}
We call $\mu\in\mathbb{C}$ an eigenvalue of \eqref{linear} with eigenfunction $\psi$, if there exist $(\lambda,\tau)\in\mathbb{R}^2_{+}$ and $\psi\in X_{\mathbb{C}}\setminus \{0\}$ solving the equation \eqref{characteristic}. Without loss of generally, we assume that $\|\psi\|_{Y_{\mathbb{C}}}=1$.
In the following, we will focus on the distribution of the eigenvalues of \eqref{linear}.
First of all, we present the following two lemmas on the prior estimates for the eigenvalue $\mu_\lambda$ and eigenfunction $\psi_\lambda$.

\begin{lemma}\label{nabla}
Assume that $\bf(H1)$ and $\bf(H2)$ holds. Then, there exists a constant $C$, such that for any $\lambda\in(\lambda_*,\lambda^*]$ and $(\mu_{\lambda},\tau_{\lambda},\psi_{\lambda})\in\mathbb{C}\times\mathbb{R}_{+}\times X_{\mathbb{C}}\setminus\{0\}$ with $\mathrm{Re}\mu_{\lambda}\geqslant0$ solving \eqref{characteristic},
$$
||\nabla\psi_{\lambda}||_{Y_{\mathbb{C}}}\leqslant C.
$$
\end{lemma}

\begin{proof}
It follows from the continuity of $\lambda\longmapsto(\xi_{\lambda},\alpha_{\lambda})$ that $\alpha_{\lambda}\in\mathbb{R}$ and $\xi_{\lambda}, u_{\lambda}\in H^2$ are bounded for any $\lambda\in[\lambda_{*},\lambda^*]$.
By the embedding theorems, we know $u_{\lambda},\xi_{\lambda}\in C^{1+\gamma}(\overline{\Omega})$ for $0<\gamma<\frac{1}{2}$.
Recall that
$$
\begin{cases}
(1+Du_{\lambda})\Delta u_{\lambda}+D\nabla u_{\lambda}\cdot \nabla u_{\lambda}+\lambda u_{\lambda}\left(m(x)-u_{\lambda}\right)=0,\\
\partial_{\nu}u_\lambda=0.
\end{cases}
$$
Since $1+Du_{\lambda}\in C^{1+\gamma}(\overline{\Omega})$ and $D\nabla u_{\lambda}\in C^{\gamma}(\overline{\Omega})$, by the regularity theory for elliptic equations, we obtain $u_{\lambda},\xi_{\lambda}\in C^{2+\gamma}(\overline{\Omega})$ and there exits a constant $C_1>0$ such that
\begin{equation}
|\xi_{\lambda}|_{2+\gamma}\leqslant C_1,\quad |u_{\lambda}|_{2+\gamma}\leqslant C_1.
\label{fangsuoC}
\end{equation}

From Lemma 3.1 in \cite{anqi}, we have that
\begin{equation}
\mathrm{Re}\{\langle \psi_{\lambda},\nabla\cdot(\psi_{\lambda}\nabla u_{\lambda}) \rangle\}
=\frac{1}{2}\langle \psi_{\lambda},\psi_{\lambda}\Delta u_{\lambda} \rangle.
\label{realFS}
\end{equation}
Taking the inner product of $\psi_{\lambda}$ with both sides of $\Delta(\mu_{\lambda},\lambda,\tau_{\lambda})\psi_{\lambda}=0$, and using $\mathrm{Re}{\mu_{\lambda}}\geqslant0$, \eqref{fangsuoC} and \eqref{realFS}, we get
\begin{equation}
\begin{split}
||\nabla\psi_{\lambda}||^2_{Y_{\mathbb{C}}}
&=-D\langle \nabla\psi_{\lambda}, u_{\lambda}\nabla\psi_{\lambda} \rangle \mathrm{Re}\{e^{-\mu_{\lambda}\tau_{\lambda}}\}
  +\frac{D}{2}\langle \psi_{\lambda},\psi_{\lambda}\Delta u_{\lambda} \rangle\\
&\quad+\langle \psi_{\lambda},[\lambda m(x)-2\lambda u_{\lambda}-\mathrm{Re}\{\mu_{\lambda}\}]\psi_{\lambda} \rangle\\
&\leqslant|D|\max \limits_{\lambda\in[\lambda_*,\lambda^*],\overline{\Omega}}\{u_{\lambda}\}||\nabla\psi_{\lambda}||^2_{Y_{\mathbb{C}}}\\
&\quad+   \left[\frac{|D|}{2}||\Delta u_{\lambda}||_{\infty}+\|\lambda m(x)-2\lambda u_{\lambda}\|_{\infty}\right]||\psi_{\lambda}||_{Y_{\mathrm{C}}}^2.
\end{split}
\end{equation}
Therefore,
$$
||\nabla\psi_{\lambda}||^2_{Y_{\mathbb{C}}}\leqslant   \frac{\tilde{C}}{1-|D|\max \limits_{\lambda\in[\lambda_*,\lambda^*],\overline{\Omega}}\{u_{\lambda}\}}||\psi_{\lambda}||_{Y_{\mathrm{C}}}^2
:=C
$$
where $\tilde{C}=\frac{|D|}{2}||\Delta u_{\lambda}||_{\infty}+\|\lambda m(x)-2\lambda u_{\lambda}\|_{\infty}$.
\end{proof}
\begin{lemma}\label{bound}
Assume that $\bf(H1)$ and $\bf(H2)$ holds. If $(\mu_{\lambda},\tau_{\lambda},\psi_{\lambda})\in\mathbb{C}\times\mathbb{R}_+\times X_{\mathbb{C}}\setminus\{0\}$ is the solution of \eqref{characteristic} with $\mathrm{Re}\mu_{\lambda}\geqslant0$, then $|\frac{\mu_{\lambda}}{\lambda-\lambda_*}|$ is bounded for $\lambda\in(\lambda_*,\lambda^*]$.
\end{lemma}

\begin{proof}
Since $(\mu_{\lambda},\tau_{\lambda},\psi_{\lambda})$ is the solution of \eqref{characteristic} with $\lambda\in(\lambda_*,\lambda^*]$, we have that
$$
\langle\psi_{\lambda},\Delta\psi_{\lambda}+D\nabla\cdot(u_{\lambda}\nabla\psi_{\lambda})e^{-\mu_{\lambda}\tau_{\lambda}}+D\nabla\cdot(\psi_{\lambda}\nabla u_{\lambda})+\lambda m(x)\psi_{\lambda}-2\lambda u_{\lambda}\psi_{\lambda}-\mu_{\lambda}\psi_{\lambda}\rangle=0.
$$
Let
\begin{equation*}
\begin{split}
A(\lambda)\psi:
&=\Delta\psi+D\nabla\cdot(u_{\lambda}\nabla\psi)+\lambda\psi[m(x)-u_{\lambda}]\\
&=\nabla\cdot[(1+Du_{\lambda})\nabla\psi]+\lambda[m(x)-u_{\lambda}]\psi.
\end{split}
\end{equation*}
Because $u_{\lambda}>0$ is the solution of $A(\lambda)\psi=0$, we know $0$ is the principal eigenvalue of $A(\lambda)$, and therefore, $\langle \psi,A(\lambda)\psi\rangle\leqslant0$ for any $\psi\in X_{\mathbb{C}}$.
From $\Delta(\mu_{\lambda},\lambda,\tau_{\lambda})\psi_{\lambda}=0$, we obtain that
\begin{equation*}
\begin{split}
0\geqslant\langle \psi_{\lambda},A(\lambda)\psi_{\lambda} \rangle&=\mu_{\lambda}
-D(e^{-\mu_{\lambda}\tau_{\lambda}}-1)\langle \psi_{\lambda},\nabla\cdot(u_{\lambda}\nabla\psi_{\lambda}) \rangle\\
&\quad -D\langle \psi_{\lambda},\nabla\cdot(\psi_{\lambda}\nabla u_{\lambda}) \rangle+\langle \psi_{\lambda},\lambda u_{\lambda}\psi_{\lambda} \rangle.
\end{split}
\end{equation*}
Since $\mathrm{Re}\mu_{\lambda}\geqslant0$,
\begin{equation*}
\begin{split}
0&\leqslant \mathrm{Re}\left(\frac{\mu_{\lambda}}{\lambda-\lambda_*}\right)\\
&\leqslant \mathrm{Re}\left\{\frac{1}{\lambda-\lambda_*}\left[
  D(e^{-\mu_{\lambda}\tau_{\lambda}}-1)\langle \psi_{\lambda},\nabla\cdot(u_{\lambda}\nabla\psi_{\lambda}) \rangle
  +D\langle \psi_{\lambda},\nabla\cdot(\psi_{\lambda}\nabla u_{\lambda}) \rangle-\langle \psi_{\lambda}, \lambda u_{\lambda}\psi_{\lambda}\rangle
  \right]\right\}\\
&=\mathrm{Re}\left\{\alpha_{\lambda}D(1-e^{-\mu_{\lambda}\tau_{\lambda}})\langle \nabla\psi_{\lambda},f_1\nabla\psi_{\lambda} \rangle
  -\alpha_{\lambda}D\langle \nabla\psi_{\lambda},\psi_{\lambda}\nabla f_1 \rangle
  -\lambda\alpha_{\lambda} \langle \psi_{\lambda}, f_1\psi_{\lambda} \rangle\right\}\\
&\leqslant\alpha_{\lambda}|D|\left[
2||f_1||_{\infty}||\nabla\psi_{\lambda}||^2_{Y_{\mathbb{C}}}+||\nabla f_1||_{\infty}||\nabla\psi_{\lambda}||_{Y_{\mathbb{C}}}||\psi_{\lambda}||_{Y_{\mathbb{C}}}
\right]+\lambda\alpha_{\lambda}||f_1||_{\infty}||\psi_{\lambda}||^2_{Y_{\mathbb{C}}}.
\end{split}
\end{equation*}
Similarly,
\begin{equation*}
\begin{split}
\left|\mathrm{Im}\left(\frac{\mu_{\lambda}}{\lambda-\lambda_*}\right)\right|
\leqslant\alpha_{\lambda}|D|\left[||f_1||_{\infty}||\nabla\psi_{\lambda}||^2_{Y_{\mathbb{C}}}+||\nabla f_1||_{\infty}||\nabla\psi_{\lambda}||_{Y_{\mathbb{C}}}||\psi_{\lambda}||_{Y_{\mathbb{C}}}\right].
\end{split}
\end{equation*}
From Lemma \ref{nabla}, it is evident that $\left|\mathrm{Re}\left(\frac{\mu_{\lambda}}{\lambda-\lambda_*}\right)\right|$ and $\left|\mathrm{Im}\left(\frac{\mu_{\lambda}}{\lambda-\lambda_*}\right)\right|$ are bounded for $\lambda\in(\lambda_*,\lambda^*]$, and so is $\left|\frac{\mu_{\lambda}}{\lambda-\lambda_*}\right|$.
\end{proof}

From Lemma \ref{bound}, we have the following stability result of \eqref{model} for $\tau=0$.
\begin{proposition}{\label{0W}}
 Assume that $\bf(H1)$ and $\bf(H2)$ holds. If $\tau=0$, then there exists $\overline{\lambda}\in(\lambda_*,\lambda^*]$, such that all the eigenvalues of \eqref{linear} have negative real parts for any $\lambda\in(\lambda_*,\overline{\lambda}]$.
\end{proposition}
\begin{proof}
If the assertion is not valid, there exists the sequence $\{(\lambda_n,\mu_{\lambda_n},\psi_{\lambda_n})\}_{n=1}^{\infty}\subset(\lambda_*,\overline{\lambda}]\times\mathbb{C}\times X_{\mathbb{C}}\setminus\{0\}$ solving \eqref{characteristic} such that $\lim \limits_{n\rightarrow\infty}\lambda_n=\lambda_*$ and $\mathrm{Re}\mu_{\lambda_n}\geqslant0$ for any $n\geqslant1$.
Since $X_{\mathbb{C}}=\left(\mathrm{Ker}\left(\Delta+\lambda_*m(x)\right)\right)_{\mathbb{C}}\oplus\left(X_1\right)_{\mathbb{C}}$, $\|\psi\|^2_{Y_{\mathbb{C}}}=\|\phi\|^2_{Y_{\mathbb{C}}}=1$. Ignoring a scalar factor, we have
$$
\psi_{\lambda_n}=\beta_{\lambda_n}\phi+(\lambda_n-\lambda_*)z_{\lambda_n},\quad \beta_{\lambda_n}\geqslant0,\quad z_{\lambda_n}\in(X_1)_{\mathbb{C}}.
$$
Substituting $(\lambda_n,\mu_{\lambda_n},\psi_{\lambda_n})$ into \eqref{characteristic} with $\tau=0$, we have
\begin{equation}
\begin{split}
\tilde{g}_1(z_{\lambda_n},\beta_{\lambda_n},\tilde{h}_{\lambda_n},\lambda_n)
&:=[\Delta+\lambda_*m(x)]z_{\lambda_n}+\alpha_{\lambda_n}D\nabla\cdot\left(f_1\nabla[\beta_{\lambda_n}\phi+(\lambda_n-\lambda_*)z_{\lambda_n}]\right)\\
     & \quad+\alpha_{\lambda_n}D\nabla\cdot\left([\beta_{\lambda_n}\phi+(\lambda_n-\lambda_*)z_{\lambda_n}]\nabla f_1\right)\\
     & \quad+[m(x)-2\alpha_{\lambda_n}\lambda_nf_1-\tilde{h}_{\lambda_n}][\beta_{\lambda_n}\phi+(\lambda_n-\lambda_*)z_{\lambda_n}]=0,\\
\tilde{g}_2(z_{\lambda_n},\beta_{\lambda_n},\tilde{h}_{\lambda_n},\lambda_n)
&:=\left(\beta_{\lambda_{n}}^2-1\right)\|\phi\|^2_{Y_{\mathbb{C}}}+(\lambda-\lambda_*)^2\|z_{\lambda_n}\|^2_{Y_{\mathbb{C}}}=0,
\label{tau0g}
\end{split}
\end{equation}
where $|\tilde{h}_{\lambda_n}|:=|\frac{\mu_{\lambda_n}}{\lambda_n-\lambda_*}|$, which is bounded from Lemma \ref{bound}.
Note that $|\beta_{\lambda_n}|\leqslant1$ from the second equation of \eqref{tau0g}. We shall finish the proof in two steps.

Step1: show that $\{z_{\lambda_n}\}_{n=1}^{\infty}$ is bounded in $Y_{\mathbb{C}}$. Taking the inner product of $z_{\lambda_n}$ with both sides of the first equation in \eqref{tau0g} and using \eqref{realFS}, we have
\begin{equation}
\begin{split}
\|\nabla z_{\lambda_n}\|^2_{Y_{\mathbb{C}}}
\leqslant \lambda_*M_0\|z_{\lambda_n}\|^2_{Y_{\mathbb{C}}}+M_1\|z_{\lambda_n}\|_{Y_{\mathbb{C}}}
+(\lambda_n-\lambda_*)M_2\|\nabla z_{\lambda_n}\|^2_{Y_{\mathbb{C}}}+(\lambda_n-\lambda_*)M_3\|z_{\lambda_n}\|^2_{Y_{\mathbb{C}}}
\end{split}
\label{nablaz}
\end{equation}
where
\begin{equation*}
\begin{split}
M_0&=\|m(x)\|_{\infty},\\
M_1&=2\alpha_{\lambda_n}|D|\|\nabla f_1\|_{\infty}\|\nabla\phi\|_{Y_{\mathbb{C}}}
     +\alpha_{\lambda_n}|D|\left(\|f_1\|_{\infty}\|\Delta\phi\|_{Y_{\mathbb{C}}}+\|\Delta f_1\|_{\infty}\right)\\
     &\quad+\|m(x)-2\alpha_{\lambda_n}\lambda_n f_1-\tilde{h}_{\lambda_n}\|_{\infty}\|\phi\|_{Y_{\mathbb{C}}},\\
M_2&=\alpha_{\lambda_n}|D|\|f_1\|_{\infty},\\
M_3&=\alpha_{\lambda_n}\frac{|D|}{2}\|\nabla f_1\|_{\infty}+\|m(x)-2\alpha_{\lambda_n}\lambda_n f_1-\tilde{h}_{\lambda_n}\|_{\infty}.
\end{split}
\end{equation*}
Choose an integer $N_1>0$ such that $1-(\lambda_n-\lambda_*)M_2>0$
for $n>N_1$. Then, \eqref{nablaz} becomes
\begin{equation}
\begin{split}
\|\nabla z_{\lambda_n}\|^2_{Y_{\mathbb{C}}}
& \leqslant\frac{M_1}{1-(\lambda_n-\lambda_*)M_2}\|z_{\lambda_n}\|_{Y_{\mathbb{C}}}
+\frac{\lambda_*M_0+(\lambda_n-\lambda_*)M_3}{1-(\lambda_n-\lambda_*)M_2}\|z_{\lambda_n}\|^2_{Y_{\mathbb{C}}}.\\
\end{split}\label{nablaz1}
\end{equation}
Note that if $z\in (X_1)_{\mathbb{C}}$, then
\begin{equation}\label{secondvalue}
|\langle z, [\Delta+\lambda_*m(x)]z\rangle|\geqslant\tilde{\lambda}_2\|z\|^2_{Y_{\mathbb{C}}}
\end{equation}
where $\tilde{\lambda}_2$ is the second eigenvalue of operator $-[\Delta+\lambda_*m(x)]$.
Therefore, from the first equation in \eqref{tau0g}, \eqref{secondvalue} and \eqref{nablaz1},
we have
\begin{equation}\label{nablaz2}
\begin{split}
\tilde{\lambda}_2\|z_{\lambda_n}\|^2_{\mathrm{Y}_{\mathbb{C}}}
&\leqslant M_1\|z_{\lambda_n}\|_{\mathrm{Y}_{\mathbb{C}}}+(\lambda_n-\lambda_*)M_2\|\nabla z_{\lambda_n}\|^2_{\mathrm{Y}_{\mathbb{C}}}+(\lambda_n-\lambda_*)M_3\|z_{\lambda_n}\|^2_{\mathrm{Y}_{\mathbb{C}}}\\
&\leqslant\left[M_1+\frac{(\lambda_n-\lambda_*)M_1M_2}{1-(\lambda_n-\lambda_*)M_2}\right]\|z_{\lambda_n}\|_{\mathrm{Y}_{\mathbb{C}}}
        +\frac{(\lambda_n-\lambda_*)\left(\lambda_*M_0M_2+M_3\right)}{1-(\lambda_n-\lambda_*)M_2}\|z_{\lambda_n}\|_{\mathrm{Y}_{\mathbb{C}}}^2.
\end{split}
\end{equation}
Let $\tilde{N}_2>0$ be the integer such that
$
\tilde{\lambda}_2-\frac{(\lambda_n-\lambda_*)\left(\lambda_*M_0M_2+M_3\right)}{1-(\lambda_n-\lambda_*)M_2}>\frac{\tilde{\lambda}_2}{2}
$
for any $n>N_2=\max\{N_1,\tilde{N}_2\}$. It then follows from \eqref{nablaz2} that
\begin{equation*}
\|z_{\lambda_n}\|_{\mathrm{Y}_{\mathbb{C}}}\leqslant 2M_4
\end{equation*}
where $M_4=\frac{2M_1}{\tilde{\lambda}_2}+\frac{2(\lambda_n-\lambda_*)M_1M_2}{\tilde{\lambda}_2\left[1-(\lambda_n-\lambda_*)M_2\right]}$,
which implies the boundedness of $\{z_{\lambda_n}\}_{n=1}^{\infty}$ in $Y_{\mathbb{C}}$.

Step2:
Since operator $[\Delta+\lambda_*m(x)]^{-1}$ is bounded from $(\mathrm{Y}_1)_{\mathbb{C}}\rightarrow(\mathrm{X}_1)_{\mathbb{C}}$, $\{z_{\lambda_n}\}_{n=1}^{\infty}$ is bounded in $(\mathrm{X}_1)_{\mathbb{C}}$. Thus, $\{(z_{\lambda_n},\beta_{\lambda_n},\tilde{h}_{\lambda_n})\}$ is precompact in
$Y_{\mathbb{C}}\times\mathbb{R}\times\mathbb{C}$, which means that there is a subsequence
$\{(z_{\lambda_{n_k}},\beta_{\lambda_{n_k}},\tilde{h}_{\lambda_{n_k}})\}_{k=1}^{\infty}$ satisfying
$$
\{(z_{\lambda_{n_k}},\beta_{\lambda_{n_k}},\tilde{h}_{\lambda_{n_k}})\}_{k=1}^{\infty}\rightarrow(z_*,\beta_*,\tilde{h}_*) \quad \text{and} \quad \lambda_{n_k}\rightarrow\lambda_* \quad \text{as} \quad k\rightarrow\infty
$$
where $ z_*\in\mathrm{Y}_{\mathbb{C}},\tilde{h}_*\in\mathbb{C}$ and $\beta_*=1$.
Taking the limit of the equation
$$
[\Delta+\lambda_*m(x)]^{-1}\tilde{g}_1(z_{\lambda_n},\beta_{\lambda_n},\tilde{h}_{\lambda_n},\lambda_n)=0
$$
as $k\rightarrow\infty$, we see that $z_*\in(X_1)_{\mathbb{C}}$, and $(z_*,\beta_*,\tilde{h}_*)$ satisfies
\begin{equation*}
[\Delta+\lambda_*m(x)]z_{*}+2\alpha_{\lambda_*}D\nabla(\phi\nabla\phi)+m(x)\phi-2\alpha_{\lambda_*}\lambda_*\phi^2-\tilde{h}_*\phi=0.
\end{equation*}
Taking the inner product of $\phi$ on both sides of above equation, we obtain
\begin{equation*}
\alpha_{\lambda_*}(r_2-r_1)=\tilde{h}_*\int_{\Omega}\phi^2\mathrm{d}x.
\end{equation*}
From $r_1-r_2>0$, we know $\tilde{h}_*<0$. This contradicts with the fact that $\mathrm{Re}\tilde{h}_*=\lim \limits_{n\rightarrow\infty}\mathrm{Re}\tilde{h}_{\lambda_n}\geqslant0.$
\end{proof}

From the proof of Theorem \ref{0W}, $\mu=0$ can not be a root of the characteristic equation \eqref{characteristic}.

\begin{proposition}\label{TW1}
 Assume that $\bf(H1)$, $\bf(H2)$ and $\bf(H3)$ hold, then there exists $\tilde{\lambda}\in(\lambda_*,\overline{\lambda}]$ such that all the eigenvalues of \eqref{linear} have negative real parts for any $\lambda\in(\lambda_*,\tilde{\lambda}]$ and $\tau>0$.
\end{proposition}
\begin{proof}
If the assertion is not valid, there exists the sequence $\{(\lambda_n,\mu_{\lambda_n},\psi_{\lambda_n})\}_{n=1}^{\infty}\subset(\lambda_*,\tilde{\lambda}]\times\mathbb{C}\times X_{\mathbb{C}}\setminus\{0\}$ such that $\lim \limits_{n\rightarrow\infty}\lambda_n=\lambda_*$,
$\mathrm{Re}\mu_{\lambda_n}\geqslant0$ for any $n\geqslant1$
and $\|\psi_{\lambda_n}\|^2_{Y_{\mathbb{C}}}=\|\phi\|^2_{Y_{\mathbb{C}}}$.
By a similar argument in Proposition \ref{0W}, we have that
\begin{equation}
\begin{split}
\hat{g}_1(z_{\lambda_n},\beta_{\lambda_n},\tilde{h}_{\lambda_n},\tau_{\lambda_n},\lambda_n)
&:=[\Delta+\lambda_*m(x)]z_{\lambda_n}+\alpha_{\lambda_n}D\nabla\cdot\left([\beta_{\lambda_n}\phi+(\lambda_n-\lambda_*)z_{\lambda_n}]\nabla f_1\right)\\
     & \quad+\alpha_{\lambda_n}D\nabla\cdot\left(f_1\nabla[\beta_{\lambda_n}\phi+(\lambda_n-\lambda_*)z_{\lambda_n}]\right)
     e^{-(\lambda_n-\lambda_*)\tilde{h}_{\lambda_n}\tau_{\lambda_n}}\\
     & \quad+[m(x)-2\alpha_{\lambda_n}\lambda_nf_1-\tilde{h}_{\lambda_n}][\beta_{\lambda_n}\phi+(\lambda_n-\lambda_*)z_{\lambda_n}]=0,\\
\hat{g}_2(z_{\lambda_n},\beta_{\lambda_n},\tilde{h}_{\lambda_n},\tau_{\lambda_n},\lambda_n)
&:=\left(\beta_{\lambda_{n}}^2-1\right)\|\phi\|^2_{Y_{\mathbb{C}}}+(\lambda-\lambda_*)^2\|z_{\lambda_n}\|^2_{Y_{\mathbb{C}}}=0,
\label{Fg}
\end{split}
\end{equation}
and $\{z_{\lambda_n}\}_{n=1}^{\infty}$ is bounded in $Y_{\mathbb{C}}$, since $|e^{-(\lambda_n-\lambda_*)\tilde{h}_{\lambda_n}\tau_{\lambda_n}}|\leqslant1$. Note that operator $[\Delta+\lambda_*m(x)]^{-1}$ is bounded from $(\mathrm{Y}_1)_{\mathbb{C}}\rightarrow(\mathrm{X}_1)_{\mathbb{C}}$, then $\{z_{\lambda_n}\}_{n=1}^{\infty}$ is bounded in $(\mathrm{X}_1)_{\mathbb{C}}$, which implies that $\{(z_{\lambda_n},\beta_{\lambda_n},\tilde{h}_{\lambda_n},e^{-(\lambda_n-\lambda_*)\tau_{\lambda_n}(\mathrm{Re}\tilde{h}_{\lambda_n})},
e^{-(\lambda_n-\lambda_*)\tau_{\lambda_n}(\mathrm{Im}\tilde{h}_{\lambda_n})i})\}_{n=1}^{\infty}$ is precompact in
$Y_{\mathbb{C}}\times\mathbb{R}^3\times\mathbb{C}$.
Thus, there is a subsequence
$$
\{(z_{\lambda_{n_k}},\beta_{\lambda_{n_k}},\tilde{h}_{\lambda_{n_k}},e^{-(\lambda_{n_k}-\lambda_*)\tau_{\lambda_{n_k}}(\mathrm{Re}\tilde{h}_{\lambda_{n_k}})},
e^{-(\lambda_{n_k}-\lambda_*)\tau_{\lambda_{n_k}}(\mathrm{Im}\tilde{h}_{\lambda_{n_k}})i})\}_{k=1}^{\infty}
$$
which is convergent to $(z^*,\beta^*,\tilde{h}^*,\delta^*,e^{-i\theta^*})$, as $k\rightarrow\infty$, where
$$
\beta^*=1, z^*\in Y_{\mathbb{C}}, \tilde{h}^*\in\mathbb{C}\left(\mathrm{Re}\tilde{h}^*\geqslant0\right), \theta^*\in[0,2\pi), \delta^*\in[0,1].
$$
Taking the limit of the equation $[\Delta+\lambda_*m(x)]^{-1}\hat{g}_1(z_{\lambda_{n_k}},\beta_{\lambda_{n_k}},\tilde{h}_{\lambda_{n_k}},\tau_{\lambda_{n_k}},\lambda_{n_k})=0$ as $k\rightarrow\infty$, we have that $z^*\in(\mathrm{X_1})_\mathbb{C}$ and $(z^*,\beta^*,\tilde{h}^*,\delta^*,\theta^*)$ satisfying
$$
[\Delta+\lambda_*m(x)]z^*+\alpha_{\lambda_*}D\nabla\cdot(\phi\nabla\phi)\delta^*e^{-i\theta^*}
+\alpha_{\lambda_*}D\nabla\cdot(\phi\nabla\phi)+[m(x)-2\lambda_*\alpha_{\lambda_*}\phi-\tilde{h}^*]\phi=0,
$$
from which we further have
$$
\alpha_{\lambda_*}\delta^*r_2e^{-i\theta^*}-\alpha_{\lambda_*}r_1-\tilde{h}^*\int_{\Omega}\phi^{2}(x)\mathrm{d}x=0.
$$
By separating the real part and the imaginary part, we arrive at
\begin{equation}
\begin{cases}
\alpha_{\lambda_*}\delta^*r_2\cos\theta^*-\alpha_{\lambda_*}r_1=\mathrm{Re}\{\tilde{h}^*\}\int_{\Omega}\phi^{2}(x)\mathrm{d}x,\\
-\alpha_{\lambda_*}\delta^*r_2\sin\theta^*=\mathrm{Im}\{\tilde{h}^*\}\int_{\Omega}\phi^{2}(x)\mathrm{d}x.
\end{cases}
\label{a1}
\end{equation}
From $\bf(H1)$ and $\bf(H3)$, we know
$$
r_1>\max\{r_2,-r_2\}\geqslant0 \quad \text{and} \quad -1<\frac{r_2}{r_1}<1.
$$
However, it follows from the first equation of \eqref{a1} that
$$
\delta^*\frac{r_2}{r_1}\cos\theta^*\geqslant1
$$
where $0\leqslant\delta^*\leqslant1$, which is a contradiction.
\end{proof}

Now, we are about to examine if there exists the pure imaginary eigenvalues of \eqref{linear} for $\lambda\in(\lambda_*,\overline{\lambda}]$, when ${\bf(H3)}$ is violated. Suppose that $\mu=i\omega$ is an eigenvalue of equation \eqref{linear} with eigenfunction $\psi\in X_{\mathbb{C}}\setminus \{0\}$, where $\omega=h(\lambda-\lambda_*)>0$ and $\|\psi\|^2_{Y_{\mathbb{C}}}=\|\phi\|^2_{Y_{\mathbb{C}}}=1$. Ignoring a scalar factor, we have that
\begin{equation}
\begin{split}
&\psi=\beta\phi+(\lambda-\lambda_*)z,\quad\beta\geqslant0,\quad z\in X_{\mathbb{C}},\\
&\|\psi\|^2_{Y_{\mathbb{C}}}=\beta^2\|\phi\|^2_{Y_{\mathbb{C}}}+(\lambda-\lambda_*)^2\|z\|^2_{Y_{\mathbb{C}}}=\|\phi\|^2_{Y_{\mathbb{C}}}.
\end{split}
\label{phiFJ}
\end{equation}
Substituting \eqref{phiFJ} into \eqref{characteristic}, we obtain
\begin{equation}
\begin{split}
g_1(z,\beta,h,\theta,\lambda)
&:=[\Delta+\lambda m(x)]z+\alpha_{\lambda}D\nabla\cdot\left(f_1\nabla(\beta\phi+(\lambda-\lambda_*)z)\right)e^{-i\theta}\\
&\quad+\alpha_{\lambda}D\nabla\cdot\left((\beta\phi+(\lambda-\lambda_*)z)\nabla f_1\right)\\
&\quad +[m(x)-2\lambda\alpha_{\lambda} f_1-hi][\beta\phi+(\lambda-\lambda_*)z]=0,\\
g_2(z,\beta,h,\theta,\lambda)
&:=(\beta^2-1)\|\phi\|^2_{Y_{\mathbb{C}}}+(\lambda-\lambda_*)^2\|z\|^2_{Y_{\mathbb{C}}}=0
\label{g}
\end{split}
\end{equation}
where $\theta=\omega\tau$. If there exists $(z,\beta,h,\theta,\lambda)\in(X_1)_{\mathbb{C}}\times\mathbb{R}^2_+\times[0,2\pi)\times\mathbb{R}_+$ solving \eqref{g}, then $\mu=i\omega=ih(\lambda-\lambda_*)$ is an eigenvalue of \eqref{linear} with $(\lambda,\tau)=(\lambda,\tau_n)$ and $\psi=\beta\phi+(\lambda-\lambda_*)z$, where
$$
\tau_n=\frac{\theta+2n\pi}{\omega},\quad n=0,1,2\cdots.
$$
Define $G: (X_1)_{\mathbb{C}}\times\mathbb{R}^2\times[0,2\pi)\times\mathbb{R}\rightarrow Y_{\mathbb{C}}\times\mathbb{R}$ by $G=(g_1,g_2)$.

\begin{lemma}\label{G}
Assume that $\bf(H1)$, $\bf(H2)$ and $\bf(H4)$ hold. Then, the equation
\begin{equation}
\begin{cases}
G(z,\beta,h,\theta,\lambda_*)=0,\\
h,\beta\geqslant0,\quad \theta\in[0,2\pi)
\end{cases}
\label{a3}
\end{equation}
has a unique solution $(z_{\lambda_*},\beta_{\lambda_*},h_{\lambda_*},\theta_{\lambda_*})$, where
$$
\beta_{\lambda_*}=1, h_{\lambda_*}=\sqrt{\frac{r_1+r_2}{r_2-r_1}}\cdot\frac{\int_{\Omega}m(x)\phi^2(x)\mathrm{d}x}{\int_{\Omega}\phi^2(x)\mathrm{d}x},
\theta_{\lambda_*}=\arccos\left(\frac{r_1}{r_2}\right),
$$
and $z_{\lambda_*}\in(X_1)_{\mathbb{C}}$ is the unique solution of
$$
[\Delta+\lambda_*m(x)]z+\alpha_{\lambda_*}D\nabla\cdot(\phi\nabla\phi)e^{-i\theta^*}+\alpha_{\lambda_*}D\nabla\cdot(\phi\nabla\phi)
+[m(x)-2\alpha_{\lambda_*}\lambda_*\phi-h_{\lambda_*i}]\phi=0.
$$
\end{lemma}

\begin{proof}
Notice that $g_2=0$ if and only if $\beta=\beta_{\lambda_*}=1.$ Then rewrite $g_1$ with $\beta=1, \lambda=\lambda_*$ as
\begin{equation}
\begin{split}
g_1(z,1,h,\theta,\lambda_*)
&=[\Delta+\lambda_*m(x)]z+\alpha_{\lambda_*}D\nabla\cdot(\phi\nabla\phi)e^{-i\theta}\\
&\quad+\alpha_{\lambda_*}D\nabla\cdot(\phi\nabla\phi)+[m(x)-2\alpha_{\lambda_*}\lambda_*\phi-hi]\phi=0.
\end{split}
\label{a5}
\end{equation}
Taking the inner product of $\phi$, we have
$$
\alpha_{\lambda_*}r_2e^{-i\theta}-\alpha_{\lambda_*}r_1-hi\int_{\Omega}\phi^2\mathrm{d}x=0,
$$
from which, we get
\begin{equation}
\begin{cases}
\alpha_{\lambda_*}r_2cos\theta-\alpha_{\lambda_*}r_1=0,\\
-\alpha_{\lambda_*}r_2sin\theta-h\int_{\Omega}\phi^2\mathrm{d}x=0.
\end{cases}
\label{a2}
\end{equation}
Using$\bf(H1)$ and $\bf(H4)$, we know
$$
r_2<\min\{r_1,-r_1\}\leqslant0 \quad \text{and} \quad -1<\frac{r_1}{r_2}<1.
$$
It then follows from \eqref{a2} that
$$
\theta=\theta_{\lambda_*}=\arccos\left(\frac{r_1}{r_2}\right),
h=h_{\lambda_*}=\sqrt{\frac{r_1+r_2}{r_2-r_1}}\cdot\frac{\int_{\Omega}m(x)\phi^2(x)\mathrm{d}x}{\int_{\Omega}\phi^2(x)\mathrm{d}x}.
$$
Substituting $\theta_{\lambda_*}$ and $h_{\lambda_*}$ into \eqref{a5}, we can derive the equation for $z_{\lambda_*}$.
Therefore, \eqref{a3} has a unique solution $(z_{\lambda_*},\beta_{\lambda_*},h_{\lambda_*},\theta_{\lambda_*})$.
\end{proof}

\begin{proposition}\label{unique}
Assume that $\bf(H1)$, $\bf(H2)$ and $\bf(H4)$ hold. Then there exist  $\tilde{\lambda}^*\in(\lambda_*,\overline{\lambda}]$ and a
continuously differentiable mapping $\lambda\mapsto(z_{\lambda},\beta_{\lambda},h_{\lambda},\theta_{\lambda})$ from $[\lambda_*,\tilde{\lambda}^*]$ to $(X_1)_{\mathbb{C}}\times\mathbb{R}^2\times[0,2\pi)$ such that $G(z_{\lambda},\beta_{\lambda},h_{\lambda},\theta_{\lambda},\lambda)=0$. Moreover,
\begin{equation}
\begin{cases}
G(z,\beta,h,\theta,\lambda)=0,\\
h,\beta\geqslant0,\quad \theta\in[0,2\pi)
\end{cases}
\label{a4}
\end{equation}
has a unique solution $(z_{\lambda},\beta_{\lambda},h_{\lambda},\theta_{\lambda})$ for $\lambda\in[\lambda_*,\tilde{\lambda}^*]$.
\end{proposition}

\begin{proof}
Define the operator $T=(T_1,T_2): (X_1)_{\mathbb{C}}\times\mathbb{R}^2\times[0,2\pi)\rightarrow Y_{\mathbb{C}}\times\mathbb{R}$ by
$$
T=D_{(z,\beta,h,\theta)}G(z_{\lambda_*},\beta_{\lambda_*},h_{\lambda_*},\theta_{\lambda_*},\lambda_*).
$$
Then,
\begin{equation*}
\begin{split}
T_1(\epsilon,\kappa,\chi,\varsigma)
&=[\Delta+\lambda_*m(x)]\epsilon-i\phi\chi-i\alpha_{\lambda_*}D\nabla\cdot(\phi\nabla\phi)e^{-i\theta_{\lambda_*}}\varsigma\\
&\quad +[\alpha_{\lambda_*}D\nabla\cdot(\phi\nabla\phi)e^{-i\theta_{\lambda_*}}+\alpha_{\lambda_*}D\nabla\cdot(\phi\nabla\phi)
+[m(x)-2\lambda_*\alpha_{\lambda_*}\phi-h_{\lambda_*}i]\phi]\kappa,\\
T_2(\kappa)
&=2\kappa\|\phi\|^2_{Y_{\mathbb{C}}}.
\end{split}
\end{equation*}

We firstly prove that $T$ is a bijective mapping from $(X_1)_{\mathbb{C}}\times\mathbb{R}^2\times[0,2\pi)$ to $Y_{\mathbb{C}}\times\mathbb{R}$. To this end, it suffices to verify
that $T$ is injective. If $T(\epsilon,\kappa,\chi,\varsigma)=0$, then $T_2(\kappa)=0$, which implies that $\kappa=0$. Substituting $\kappa=0$ into $T_1(\epsilon,\kappa,\chi,\varsigma)$, we have
$$
(\Delta+\lambda_*m(x))\epsilon-i\phi\chi-i\alpha_{\lambda_*}D\nabla\cdot(\phi\nabla\phi)e^{-i\theta_{\lambda_*}}\varsigma=0,
$$
and hence
$$
-i\chi\int_{\Omega}\phi^2\mathrm{d}x-i\alpha_{\lambda_*}r_2(\cos\theta_{\lambda_*}-i\sin\theta_{\lambda_*})\varsigma=0.
$$
By separating the real part and the imaginary part, we obtain
\begin{equation*}
\begin{cases}
-\chi\int_{\Omega}\phi^2\mathrm{d}x-\alpha_{\lambda_*}r_2\varsigma\cos\theta_{\lambda_*}=0,\\
-\alpha_{\lambda_*}r_2\varsigma\sin\theta_{\lambda_*}=0.
\end{cases}
\end{equation*}
Since $\sin\theta_{\lambda_*}\neq0$ from \eqref{a2}, we have $\varsigma=0, \chi=0$ and consequently $\epsilon=0$. Therefore, $T$ is bijective from $(X_1)_{\mathbb{C}}\times\mathbb{R}^2\times[0,2\pi)$ to $Y_{\mathbb{C}}\times\mathbb{R}$. By the implicit function theorem, there
exists a continuously differentiable mapping $\lambda\mapsto(z_{\lambda},\beta_{\lambda},h_{\lambda},\theta_{\lambda})$ from $[\lambda_*,\tilde{\lambda}^*]$ to $(X_1)_{\mathbb{C}}\times\mathbb{R}^2\times[0,2\pi)$ such that $G(z_{\lambda},\beta_{\lambda},h_{\lambda},\theta_{\lambda},\lambda)=0$.

To prove the uniqueness, we shall verify that if $z^{\lambda}\in(X_1)_{\mathbb{C}}, \beta^{\lambda}, h^{\lambda}>0, \theta^{\lambda}\in[0,2\pi)$, and
$G(z^{\lambda},\beta^{\lambda},h^{\lambda},\theta^{\lambda},\lambda)=0$, then
$$
(z^{\lambda},\beta^{\lambda},h^{\lambda},\theta^{\lambda})\rightarrow(z_{\lambda_*},\beta_{\lambda_*},h_{\lambda_*},\theta_{\lambda_*})
$$
as $\lambda\rightarrow\lambda_*$.
It follows from Lemma \ref{bound} and \eqref{g}, $\{h^{\lambda}\}, \{\beta^{\lambda}\}$ and $\{\theta^{\lambda}\}$ are bounded for any $\lambda\in(\lambda_*,\tilde{\lambda}^*]$ and so does $\{\alpha_{\lambda}\}$ and $\{\xi_{\lambda}\}$.
For $\lambda$ sufficiently close to $\lambda_*$, we have $1-(\lambda-\lambda_*)M_2>0$. By a similar argument as in Proposition \ref{0W}, it can be verified that $\{z_{\lambda}\}$ is bounded in $(\mathrm{X}_1)_{\mathbb{C}}$. Therefore, $\{(z^{\lambda},\beta^{\lambda},h^{\lambda},\theta^{\lambda})\}$ is precompact in
$Y_{\mathbb{C}}\times\mathbb{R}^2\times[0,2\pi)$.
Let $\{(z^{\lambda^n},\beta^{\lambda^n},h^{\lambda^n},\theta^{\lambda^n})\}$ be the sequence such that
$$
(z^{\lambda^n},\beta^{\lambda^n},h^{\lambda^n},\theta^{\lambda^n})\rightarrow(z^{\lambda_*},\beta^{\lambda_*},h^{\lambda_*},\theta^{\lambda_*})\,\,
\text{in}\,\, Y_{\mathbb{C}}\times\mathbb{R}^2\times[0,2\pi), \,\,\lambda^n\rightarrow\lambda_*\,\, \text{as}\,\, n\rightarrow\infty.
$$
Taking the limit of the equation $[\Delta+\lambda_*m(x)]^{-1}g_1(z^{\lambda^n},\beta^{\lambda^n},h^{\lambda^n},\theta^{\lambda^n},\lambda^n)=0$ as $n\rightarrow\infty$, we see
$$
(z^{\lambda^n},\beta^{\lambda^n},h^{\lambda^n},\theta^{\lambda^n})\rightarrow(z^{\lambda_*},\beta^{\lambda_*},h^{\lambda_*},\theta^{\lambda_*})
\,\,\text{in}\,\, (X_1)_{\mathbb{C}}\times\mathbb{R}^2\times[0,2\pi) \,\,\text{as}\,\, n\rightarrow\infty,
$$
and $G(z^{\lambda_*},\beta^{\lambda_*},h^{\lambda_*},\theta^{\lambda_*},\lambda_*)=0$. It follows from Lemma \ref{G} that
$$
(z^{\lambda_*},\beta^{\lambda_*},h^{\lambda_*},\theta^{\lambda_*})=(z_{\lambda_*},\beta_{\lambda_*},h_{\lambda_*},\theta_{\lambda_*}).
$$
This completes the proof.
\end{proof}

\begin{corollary}\label{cor}
Assume that $\bf(H1)$, $\bf(H2)$ and $\bf(H4)$ hold, then for each fixed $\lambda\in(\lambda_*,\tilde{\lambda}^*]$, $\mu=i\omega, \omega>0$ is an eigenvalue of equation \eqref{linear} if and only if
$$
\omega=\omega_{\lambda}=h_{\lambda}(\lambda-\lambda_*),\quad \tau=\tau_{n}=\frac{\theta_{\lambda}+2n\pi}{\omega_{\lambda}},\quad n=0,1,2\cdots
$$
and
$$
\psi=c\psi_{\lambda},\quad \psi_{\lambda}=\beta_{\lambda}\phi+(\lambda-\lambda_*)z_{\lambda}
$$
where $c$ is a nonzero constant and $z_{\lambda},\beta_{\lambda},h_{\lambda},\theta_{\lambda}$ are defined as in Theorem \ref{unique}.
\end{corollary}

In the following, we will prove the transversality condition of the pure imaginary eigenvalue $i\omega$. From \eqref{characteristic}, it can be seen that any eigenvalues $\mu$ of \eqref{linear} with $(\lambda,\tau,\psi)$ in $\mathbb{R}_+^2\times X_{\mathbb{C}}\setminus\{0\}$ must satisfy
\begin{equation}\label{Q}
\begin{split}
\emph{Q}(\lambda,\mu,\tau)
&=\int_{\Omega}\psi\Delta\psi\mathrm{d}x+e^{-\mu\tau}\int_{\Omega}\psi[D\nabla\cdot(u_{\lambda}\nabla\psi)]\mathrm{d}x\\
&\quad+\int_{\Omega}\psi[D\nabla\cdot(\psi\nabla u_{\lambda})+\lambda m(x)\psi-2\lambda u_{\lambda}\psi-\mu\psi]\mathrm{d}x\\
&=0.
\end{split}
\end{equation}

\begin{proposition}\label{transcon}
Assume $\bf(H1)$, $\bf(H2)$ and $\bf(H4)$ hold. There exists the neighbourhood $O_n\times D_n\subset\mathbb{R}\times\mathbb{C}$ of $(\tau_n,i\omega)$ and continuous differential mapping $\tau\mapsto\mu(\tau)$ from $O_n$ to $D_n$, such that $\mu(\tau_n)=i\omega_{\lambda}$ and $\emph{Q}(\lambda,\mu(\tau),\tau)=0$. Moreover,
$$
\frac{\mathrm{d}\mathrm{Re}\left(\mu(\tau_n)\right)}{\mathrm{d}\tau}>0,\quad n=0,1,2,\cdots.
$$
\end{proposition}

\begin{proof}
We shall finish the proof in two steps.

Step1:
For convenience, we define
$$
S_n(\lambda):=\int_{\Omega}\psi^2_{\lambda}\mathrm{d}x+\tau_nDe^{-i\theta_{\lambda}}\int_{\Omega}\psi_{\lambda}\nabla\cdot(u_{\lambda}\nabla\psi_{\lambda})\mathrm{d}x,\quad n=0,1,2,\cdots.
$$
It follows from Proposition \ref{unique} and Corollary \ref{cor} that $\psi_{\lambda}\rightarrow\phi, \theta_{\lambda}\rightarrow\theta_{\lambda_*}$ as $\lambda\rightarrow\lambda_*$. Since $\sin\theta_{\lambda_*}\neq0$ from \eqref{a2}, then
\begin{equation*}
\begin{split}
\lim\limits_{\lambda\rightarrow\lambda_*}S_n
&=\int_{\Omega}\phi^2\mathrm{d}x+\alpha_{\lambda_*}\frac{\theta_{\lambda_*}+2n\pi}{h_{\lambda_*}}r_2e^{-i\theta_{\lambda_*}}\\
&=\int_{\Omega}\phi^2\mathrm{d}x+\alpha_{\lambda_*}\frac{\theta_{\lambda_*}+2n\pi}{h_{\lambda_*}}r_2\cos\theta_{\lambda_*}
-i\alpha_{\lambda_*}\frac{\theta_{\lambda_*}+2n\pi}{h_{\lambda_*}}r_2\sin\theta_{\lambda_*}\\
&\neq0.
\end{split}
\end{equation*}
Therefore, $S_n(\lambda)\neq0$ with $\lambda\in(\lambda_*, \tilde{\lambda}^*]$.
Then from \eqref{Q}, we see that
\begin{equation*}
\begin{split}
\frac{\partial{\emph{Q}}}{\partial{\mu}}(\lambda,i\omega_{\lambda},\tau_n)
&=\int_{\Omega}\psi_{\lambda}[D\nabla\cdot(u_{\lambda}\nabla\psi_{\lambda})]\mathrm{d}x\left(-\tau_ne^{-i\theta_{\lambda}}\right)-\int_{\Omega}\psi^2_{\lambda}\mathrm{d}x\\
&=-[\int_{\Omega}\psi^2_{\lambda}\mathrm{d}x+\tau_nDe^{-i\theta_{\lambda}}\int_{\Omega}\psi_{\lambda}\nabla\cdot(u_{\lambda}\nabla\psi_{\lambda})\mathrm{d}x]\\
&=-S_n\neq0.
\end{split}
\end{equation*}
From Implicit function theorem, there exists the neighbourhood $O_n\times D_n\subset\mathbb{R}\times\mathbb{C}$ of $(\tau_n,i\omega)$ and continuous differential mapping $\tau\mapsto\mu(\tau)$ from $O_n$ to $D_n$, such that $\mu(\tau_n)=i\omega_{\lambda}$ and $\emph{Q}(\lambda,\mu(\tau),\tau)=0$. The equipped eigenvalue function of $\mu(\tau)$ is $\psi(\tau)$ and $\psi(\tau_n)=\psi_{\lambda}$.

Step2:
show the  transversality condition. Differential the equation $\Delta(\lambda,\mu(\tau),\tau)\psi(\tau)=0$ with respect to $\tau=\tau_n$, we arrive that
\begin{equation}\label{tran}
\Delta(\lambda,i\omega_{\lambda},\tau_n)\frac{\mathrm{d}\psi(\tau_n)}{\mathrm{d}\tau}
+\frac{\mathrm{d}\mu(\tau_n)}{\mathrm{d}\tau}[D\nabla\cdot(u_{\lambda}\nabla\psi_{\lambda})(-\tau_ne^{-i\theta_{\lambda}})-\psi_{\lambda}]
-i\omega_{\lambda}e^{-i\theta_{\lambda}}D\nabla\cdot(u_{\lambda}\nabla\psi_{\lambda})=0.
\end{equation}
Since
$$
\langle \Delta(\lambda,i\omega_{\lambda},\tau_n)\frac{\mathrm{d}\psi(\tau_n)}{\mathrm{d}\tau}, \overline{\psi}_{\lambda} \rangle
=\langle \frac{\mathrm{d}\psi(\tau_n)}{\mathrm{d}\tau},\Delta(\lambda,-i\omega_{\lambda},\tau_n) \overline{\psi}_{\lambda} \rangle
=0.
$$
Taking inner product of both sides of \eqref{tran} with $\overline{\psi}_{\lambda}$, we obtain
\begin{equation*}
\begin{split}
&\quad \frac{\mathrm{d}\mu(\tau_n)}{\mathrm{d}\tau}
   =\frac{-i\omega_{\lambda}De^{-i\theta_{\lambda}}\int_{\Omega}\nabla\cdot(u_{\lambda}\nabla\psi_{\lambda})\psi_{\lambda}\mathrm{d}x}
      {\int_{\Omega}\psi^2_{\lambda}\mathrm{d}x+\tau_nDe^{-i\theta_{\lambda}}\int_{\Omega}\nabla\cdot(u_{\lambda}\nabla\psi_{\lambda})\psi_{\lambda}\mathrm{d}x}\\
&=\frac{-i\omega_{\lambda}D}{|S_n|^2}
     \left(-e^{-i\theta_{\lambda}}\int_{\Omega}u_{\lambda}|(\nabla\psi_{\lambda})^2|e^{i\tilde{\gamma}_{\lambda}}\mathrm{d}x
      \overline{\int_{\Omega}|\psi^2_{\lambda}|e^{i{\gamma}_{\lambda}}\mathrm{d}x}
      +\tau_nD\left|\int_{\Omega}\nabla\cdot(u_{\lambda}\nabla\psi_{\lambda})\psi_{\lambda}\mathrm{d}x\right|^2\right)  \\
&=\frac{\omega_{\lambda}D\sin\left(\theta_{\lambda}+\gamma_{\lambda}-\tilde{\gamma}_{\lambda}\right)}{|S_n|^2}
       \int_{\Omega}u_{\lambda}|(\nabla\psi_{\lambda})^2|\mathrm{d}x{\int_{\Omega}|\psi^2_{\lambda}|\mathrm{d}x} \\
&\quad +\frac{i\omega_{\lambda}D}{|S_n|^2}\left[\cos\left(\theta_{\lambda}+\gamma_{\lambda}-\tilde{\gamma}_{\lambda}\right)
           \int_{\Omega}u_{\lambda}|(\nabla\psi_{\lambda})^2|\mathrm{d}x{\int_{\Omega}|\psi^2_{\lambda}|\mathrm{d}x}
        -\tau_nD\left|\int_{\Omega}\nabla\cdot(u_{\lambda}\nabla\psi_{\lambda})\psi_{\lambda}\mathrm{d}x\right|^2\right]
\end{split}
\end{equation*}
where
$$
\gamma_{\lambda}=\mathrm{Arg}(\psi_{\lambda}^2),\quad\quad  \tilde{\gamma}_{\lambda}=\mathrm{Arg}\left[(\nabla\psi_{\lambda})^2\right],
   \quad\quad -\pi<\gamma_{\lambda}, \tilde{\gamma}_{\lambda}\leqslant\pi.
$$
From the expressions of $u_{\lambda},\psi_{\lambda}$ and $\tau_n$, we see that
$$
\theta_{\lambda}\rightarrow\theta_{\lambda_*},\quad \psi_{\lambda}\rightarrow\phi, \quad h_{\lambda}\rightarrow h_{\lambda_*}>0,
\quad \gamma_{\lambda}, \tilde{\gamma}_{\lambda}\rightarrow 0
\quad \text{as} \quad \lambda\rightarrow\lambda^*.
$$
From $\bf(H1)$ and $\bf(H4)$, we have $r_2<0$. It follows from \eqref{a2} that $\sin\theta_{\lambda_*}>0.$ Moreover, we obtain
\begin{equation*}
\begin{split}
\lim\limits_{\lambda\rightarrow\lambda_*}&\frac{1}{(\lambda-\lambda_*)^2}\frac{\mathrm{d}\mathrm{Re}\left(\mu(\tau_n)\right)}{\mathrm{d}\tau}\\
&=\frac{1}{\lim\limits_{\lambda\rightarrow\lambda_*}|S_n|^2}
\left(-\sin\theta_{\lambda_*}h_{\lambda_*}\alpha_{\lambda_*}r_2\int_{\Omega}\phi^2\mathrm{d}x\right)>0.
\end{split}
\end{equation*}
Therefore, for $\lambda\in(\lambda_*, \tilde{\lambda}^*]$
$$
\frac{\mathrm{d}\mathrm{Re}\left(\mu(\tau_n)\right)}{\mathrm{d}\tau}>0.
$$
\end{proof}

Theorem \ref{conclusion} $(1)-(i)$ is a consequence of Proposition \ref{0W} and \ref{TW1}; and Theorem \ref{conclusion} $(1)-(ii)$ follows directly from Proposition \ref{unique} and \ref{transcon}. The proof of the second statement $(2)$ of Theorem \ref{conclusion} is omitted, since it is analogous to the analysis of Theorem \ref{conclusion} $(1)$.

\section{The proof of Theorem \ref{conclusion2}}

Under the assumption ${\bf(A2)}$, 
let $\overline{m}:=\frac{1}{|\Omega|}\int_{\Omega}m(x)\mathrm{d}x$. Then, the following equation has a unique solution for $\rho_m$
\begin{equation}
\label{rho}
\begin{cases}
\left(1+D\overline{m}\right)\Delta\rho_m(x)+\overline{m}(m(x)-\overline{m})=0,&x\in\Omega,\\
\int_{\Omega}\rho_{m}(x)\mathrm{d}x=0,\\
\partial_{\nu}\rho_m(x)=0,&x\in\partial\Omega.
\end{cases}
\end{equation}
Similarly, denote
\begin{equation}
\label{C(m)}
C(m):=\frac{(\overline{m}D+1)\int_{\Omega}(\nabla\rho_m)^2\mathrm{d}x}{\overline{m}^2|\Omega|}.
\end{equation}
Then,  the following equation will also have a unique solution for $\gamma_m$
\begin{equation}
\label{gamma}
\begin{cases}
\left(1+D\overline{m}\right)\Delta\gamma_m(x)+f_{\gamma}(x)=0,&x\in\Omega,\\
\int_{\Omega}\gamma_{m}(x)\mathrm{d}x=0,\\
\partial_{\nu}\gamma_m(x)=0,&x\in\partial\Omega,
\end{cases}
\end{equation}
where
$$
f_{\gamma}(x)=D(\nabla\rho_m)^2+D(\rho_{m}+C(m))\Delta\rho_m+(\rho_m+C(m))(m(x)-2\overline{m}).
$$
Furthermore, from the regular theory for elliptic equations, we know $\rho_m, \gamma_m\in C^{2+\alpha}(\overline{\Omega})$ for $\alpha\in(0,1)$. Similar to Proposition 3.1 in \cite{He}, we construct the upper and lower solutions of \eqref{model} and have the following conclusions on the local expression of steady state.

\begin{proposition}\label{locexis}
Assume that $\bf(A2)$ and $\bf(O)$ hold. Then, there exists a small constant $\lambda^*>0$ such that, for $\lambda\in(0,\lambda^*]$, the positive steady state $u_{\lambda}$ of system \eqref{model} can be locally represented as
\begin{equation}
\label{steady-state}
u_{\lambda}=\overline{m}+\lambda(\rho_m(x)+C(m))+\lambda^2(\gamma_m(x)+K(m))+o(\lambda^2)=:\overline{m}+\lambda f_1(x)
\end{equation}
where
\begin{equation*}\label{Km}
K(m)=\frac{\int_{\Omega}\gamma_m(m(x)-\overline{m})\mathrm{d}x-\int_{\Omega}(\rho_m+C(m))^2\mathrm{d}x}
     {\overline{m}|\Omega|}.
\end{equation*}
Moreover, $u_{\lambda}\in C^{2+\alpha}(\overline{\Omega})$ for $\alpha\in(0,1)$.
\end{proposition}

\begin{proof}
The definition of $K(m)$ guarantees the existence and uniqueness of the solution for $\theta_3$ to the problem
\begin{equation*}
\begin{cases}
\left(1+D\overline{m}\right)\Delta \theta_3(x)+D(\rho_m+C(m))\Delta\gamma_m+f_3(x)=0, &x\in\Omega,\\
\int_{\Omega}\theta_{3}(x)\mathrm{d}x=0,\\
\partial{\nu}\theta_3=0,&x\in\partial{\Omega},
\end{cases}
\end{equation*}
where
$$
f_3(x)=D(\gamma_m+K(m))\Delta\rho_m+(\gamma_m+K(m))(m(x)-2\overline{m})+2D\nabla\rho_m\nabla\gamma_m-(\rho_m+C(m))^2.
$$
Now, following system has a unique solution for $\theta_4$.
\begin{equation*}\label{theta4}
\begin{cases}
\left(1+D\overline{m}\right)\Delta\theta_4(x)+D(\rho_m+C(m))\Delta\theta_3+f_4(x)=0, &x\in\Omega,\\
\int_{\Omega}\theta_{4}(x)\mathrm{d}x=0,\\
\partial{\nu}\theta_4=0,&x\in\partial{\Omega},
\end{cases}
\end{equation*}
where
\begin{equation*}
\begin{split}
f_4(x)=&D(\nabla \gamma_m)^2+D(\gamma_m+K(m))\Delta\gamma_m+D(\theta_3+c_3)\Delta\rho_m+2D\nabla\theta_3\nabla\rho_m\\
&+(\theta_3+c_3)(m(x)-2\overline{m})-2(\gamma_m+K(m))(\rho_m+C(m))
\end{split}
\end{equation*}
with
$$
c_3=\frac{\int_{\Omega}\theta_3(m(x)-\overline{m})\mathrm{d}x-2\int_{\Omega}(\gamma_m(x)+K(m))(\rho_m(x)+C(x))}
    {\overline{m}|\Omega|}.
$$

Define
\begin{equation}\label{upm}
u_{\lambda}^{\pm}:=\overline{m}+\lambda(\rho_m+C(m))+\lambda^2(\gamma_m+K(m))
+\lambda^3(\theta_3+c_3)+\lambda^4\theta_4\pm\left(\overline{m}\lambda^3+\frac{\rho_m}{1+D\overline{m}}\lambda^4\right).
\end{equation}
We claim that $u_{\lambda}^{\pm}$
are a pair of upper and lower solutions to \eqref{model} for sufficiently small $\lambda>0$. Indeed, direct calculation yields
$$
\Delta u_{\lambda}^{\pm}+D\nabla\cdot( u_{\lambda}^{\pm}\nabla u_{\lambda}^{\pm})+\lambda u_{\lambda}^{\pm}(m(x)- u_{\lambda}^{\pm})
=\mp\overline{m}^2\lambda^4+o(\lambda^5)
$$
where the following equation is used
$$
\pm\frac{\Delta\rho_m}{1+D\overline{m}}\pm \frac{D\overline{m}\Delta\rho_m}{1+D\overline{m}}\pm D\overline{m}\Delta\rho_m
\pm\overline{m}(m(x)-\overline{m})\mp\overline{m}^2=\mp \overline{m}^2.
$$
For $\lambda$ sufficiently small, $0<u_{\lambda}^{-}<u_{\lambda}^+$. By $\bf(O)$, we know $\bf(H3)$ in \cite{Pao} is satisfied. It then follows from Theorem 3.1 in \cite{Pao} that \eqref{model} admints a solution $u_{\lambda}$ such that
$$
u_{\lambda}^-\leqslant u_{\lambda}\leqslant u_{\lambda}^+.
$$
From the definition of $u_{\lambda}^{\pm}$ in \eqref{upm}, $u_{\lambda}$ can be represented as \eqref{steady-state}. Moreover, we know that $u_{\lambda}\in C^{2+\alpha}(\overline{\Omega})$ by the expression of $u_{\lambda}$.
\end{proof}
For convenience, we still denote steady state as $u_{\lambda}$, which is difference from the expression in Section \ref{A1}.
Next, we shall investigate the stability of $u_\lambda$ for $\lambda$ sufficiently small. The characteristic equation associated with $u_\lambda$ is now given by
\begin{equation}
\Delta(\lambda,\mu,\tau)\psi=\Delta\psi+D\nabla\cdot(u_{\lambda}\nabla\psi)e^{-\mu\tau}+D\nabla\cdot(\psi\nabla u_{\lambda})+\lambda[m(x)-u_{\lambda}]\psi-\lambda u_{\lambda}\psi-\mu\psi.
\label{characteristic2}
\end{equation}
As in Section 3, we also need the following two estimations on $\mu_\lambda$ and $\psi_\lambda$.

\begin{lemma}\label{nabla2}
If ${\bf(A2)}$ and ${\bf(P)}$ are satisfied, then
there exists a constant $C$, such that for any $\lambda\in(0,\lambda^*)$ and $(\mu_{\lambda},\tau_{\lambda},\psi_{\lambda})\in\mathbb{C}\times\mathbb{R}_{+}\times X_{\mathbb{C}}\setminus\{0\}$ with $\mathrm{Re}\mu_{\lambda}\geqslant0$ solving \eqref{characteristic2}, we have
$$
\|\nabla\psi_{\lambda}\|_{Y_{\mathbb{C}}}\leqslant C\|\psi_{\lambda}\|_{Y_{\mathbb{C}}}.
$$
\end{lemma}
\begin{proof}
The proof is analogous to Lemma \ref{nabla}, and hence is omitted.
\end{proof}

Notice that $u_{\lambda}\rightarrow\overline{m}$ while $\lambda\rightarrow0$ and \eqref{characteristic2} degenerate into
$$
(1+D\overline{m}e^{-\mu\tau})\Delta \psi=\mu\psi.
$$
Then $\mu\rightarrow0, \psi\rightarrow\phi$ (positive constant) as $\lambda\rightarrow0$. Therefore, we have the following lemma.

\begin{lemma}\label{bound2}
Assume that ${\bf(A2)}$ and ${\bf(P)}$ hold. If $(\mu_{\lambda},\tau_{\lambda},\psi_{\lambda})\in\mathbb{C}\times\mathbb{R}_+\times X_{\mathbb{C}}\setminus\{0\}$ is the solution of \eqref{characteristic2} with $\mathrm{Re}\mu_{\lambda}\geqslant0$, 
 then $|\frac{\mu_{\lambda}}{\lambda}|$ is bounded for $\lambda\in(0,\lambda^*]$.
\end{lemma}

\begin{proof}
Since $|\frac{\mu_{\lambda}}{\lambda}|=|\frac{\overline{\mu}_{\lambda}}{\lambda}|$, we only consider $\mathrm{Im}\mu_{\lambda}\geqslant0$.
By a similar argument in Lemma \ref{bound}, we have
\begin{equation*}
\begin{split}
0\geqslant\langle \psi_{\lambda},A(\lambda)\psi_{\lambda} \rangle&=\mu_{\lambda}
-D(e^{-\mu_{\lambda}\tau_{\lambda}}-1)\langle \psi_{\lambda},\nabla\cdot(u_{\lambda}\nabla\psi_{\lambda}) \rangle\\
&\quad -D\langle \psi_{\lambda},\nabla\cdot(\psi_{\lambda}\nabla u_{\lambda}) \rangle+\langle \psi_{\lambda},\lambda u_{\lambda}\psi_{\lambda} \rangle,
\end{split}
\end{equation*}
which implies
\begin{equation*}
\frac{\mathrm{Re}\mu_{\lambda}}{\lambda}\leqslant\frac{1}{\lambda}\mathrm{Re}\left\{
  D(e^{-\mu_{\lambda}\tau_{\lambda}}-1)\langle \psi_{\lambda},\nabla\cdot(u_{\lambda}\nabla\psi_{\lambda}) \rangle
  +D\langle \psi_{\lambda},\nabla\cdot(\psi_{\lambda}\nabla u_{\lambda}) \rangle-\langle \psi_{\lambda}, \lambda u_{\lambda}\psi_{\lambda}\rangle
  \right\}.
\end{equation*}
Using $\mathrm{Re}\mu_{\lambda}\geqslant0$ and \eqref{steady-state}, we know
\begin{equation}\label{re}
\begin{split}
\frac{\mathrm{Re}\mu_{\lambda}}{\lambda}
&\leqslant\frac{D\overline{m}}{\lambda}\left(1-\mathrm{Re}\{e^{-\mu_{\lambda}\tau_{\lambda}}\}\right)\|\nabla\psi_{\lambda}\|^2_{Y_{\mathbb{C}}}
          +\left\|D(1-e^{-\mu_{\lambda}\tau_{\lambda}}) \langle \nabla\psi_{\lambda}, f_1\nabla\psi_{\lambda}  \rangle  \right\|      \\
  &\quad  +\left\|D\langle   \psi_{\lambda},\nabla\cdot(\psi_{\lambda}\nabla f_1)\right\|
          +\left\|\langle \psi_{\lambda}, u_{\lambda}\psi_{\lambda}   \rangle\right\|           \\
&\leqslant \frac{C_a}{\lambda}\left(1-\mathrm{Re}\{e^{-\mu_{\lambda}\tau_{\lambda}}\}\right)+C_b,
\end{split}
\end{equation}
where
\begin{equation*}
\begin{split}
C_a&=C^2\overline{m}|D|\| \psi_{\lambda}\|_{{Y}_{\mathbb{C}}}^2,\\
C_b&=2C^2|D|\|f_1\|_{\infty}\|\psi_{\lambda}\|^2_{Y_{\mathbb{C}}}+DC\|\nabla f_1\|_{\infty}\|\psi_{\lambda}\|^2_{Y_{\mathbb{C}}}
          +\|u_{\lambda}\|_{\infty}\|\psi_{\lambda}\|^2_{Y_{\mathbb{C}}}
\end{split}
\end{equation*}
which are positive constant. Similarly,

\begin{equation}\label{imagpart}
\frac{\mathrm{Im}\mu_{\lambda}}{\lambda}\leqslant -\frac{C_a}{\lambda}\mathrm{Im}\{e^{-\mu_{\lambda}\tau_{\lambda}}\}+C_b.
\end{equation}

Since $\mu_{\lambda}\rightarrow 0$ as $\lambda\rightarrow 0$,
for $\lambda$ sufficiently small, we have
\begin{equation}\label{fang}
\begin{split}
\sin(\tau_{\lambda}\mathrm{Im}\mu_{\lambda})\leqslant\tau_{\lambda}\mathrm{Im}\mu_{\lambda},\quad\quad\quad\quad\quad \mathrm{Im}\mu_{\lambda} &\leqslant1,\\
1-\cos(\tau_{\lambda}\mathrm{Im}\mu_{\lambda})\leqslant\frac{1}{2}(\tau_{\lambda}\mathrm{Im}\mu_{\lambda})^2,\quad
1-e^{\tau_{\lambda}\mathrm{Re}\mu_{\lambda}}&\leqslant \tau_{\lambda}\mathrm{Re}\mu_{\lambda}.
\end{split}
\end{equation}
Suppose $\tau_{\lambda}<\overline{\tau}$ for some $\overline{\tau}$. Then, we have $1-C_a \overline{\tau} >0$ by letting $\|\psi_{\lambda}\|_{Y_{\mathbb{C}}}$ to be sufficiently small. It follows from \eqref{imagpart} and \eqref{fang} that
\begin{equation*}
\begin{split}
\frac{\mathrm{Im}\mu_{\lambda}}{\lambda}
&\leqslant
     \frac{C_a}{\lambda}e^{-\tau_{\lambda}\mathrm{Re}\mu_{\lambda}}\sin(\tau_{\lambda}\mathrm{Im}\mu_{\lambda})+C_b  \\
&\leqslant \frac{C_a}{\lambda}\overline{\tau}\mathrm{Im}\mu_{\lambda}+C_b,  \\
\end{split}
\end{equation*}
which implies
$$
\frac{\mathrm{Im}\mu_{\lambda}}{\lambda}\leqslant\frac{C_b}{1-C_a\overline{\tau}}:=M_I.
$$
Therefore, $\left|\frac{\mathrm{Im}\mu_{\lambda}}{\lambda}\right|$ is bounded for $\lambda\in(0, \lambda^*]$.
Similarly, using \eqref{re} and \eqref{fang},  we have the following estimate
\begin{equation*}
\begin{split}
0\leqslant\frac{\mathrm{Re}\mu_{\lambda}}{\lambda}
&\leqslant \frac{C_a}{\lambda}\left(1-e^{-{\tau}_{\lambda}\mathrm{Re}\mu_{\lambda}}\cos({\tau}_{\lambda}\mathrm{Im}\mu_{\lambda})\right)+C_b\\
&\leqslant \frac{C_a}{\lambda}\left[e^{-{\tau}_{\lambda}\mathrm{Re}\mu_{\lambda}}\left(1-\cos({\tau}_{\lambda}\mathrm{Im}\mu_{\lambda})\right)
            +1-e^{-{\tau}_{\lambda}\mathrm{Re}\mu_{\lambda}}\right]+C_b\\
&\leqslant \frac{C_a}{\lambda}\left[\frac{1}{2}(\tau_{\lambda}\mathrm{Im}\mu_{\lambda})^2+\tau_{\lambda}\mathrm{Re}\mu_{\lambda}\right]+C_b\\
&\leqslant M_R+\frac{C_a\overline{\tau}}{\lambda}\mathrm{Re}\mu_{\lambda}+C_b
\end{split}
\end{equation*}
where $M_R=\frac{M_IC_a\overline{\tau}^2}{2}$. Accordingly,
$$
\frac{\mathrm{Re}\mu_{\lambda}}{\lambda}\leqslant\frac{M_R+C_b}{1-C_a \overline{\tau}}.
$$
that is, $\left|\frac{\mathrm{Re}\mu_{\lambda}}{\lambda}\right|$ is also bounded for $\lambda\in(0,\lambda^*]$.
\end{proof}


\begin{proposition}{\label{0W2}}
Assume that ${\bf(A2)}$ and $\bf(P)$ hold. Then, there exists $\tilde{\lambda}\in(0,\lambda^*]$ such that all the eigenvalues of \eqref{characteristic2} have negative real parts for any $\lambda\in(0,\tilde{\lambda}]$ and $\tau\geq0$.
\end{proposition}

\begin{proof}
We consider the  following decompositions:
$$
X=K\oplus X_1,\quad Y=K\oplus Y_1,
$$
where
$$
K=\mathrm{span}\{1\},\quad X_1=\left\{y\in X:\int_{\Omega}y(x)dx=0\right\},\quad Y_1=\left\{y\in Y:\int_{\Omega}y(x)\mathrm{d}x=0\right\}.
$$
Then, the analysis is analogous to the proof in Proposition \ref{0W} and \ref{TW1} and we omit it.
\end{proof}
Theorem \ref{conclusion2} is a direct consequence of Proposition \ref{0W2}.

\section{Discussion}

In this paper, we mainly studied model \eqref{model} and examined the joint effect of memorized diffusion and spatial heterogeneity on the dynamics. In the case of ${\bf(A2)}$, it has been shown that the non-constant steady state $u_{\lambda}$ is locally asymptotically stable, as long as the memorized diffusion rate $D$ is not large, see the condition ${\bf(P2)}$. This can be expected from the results for the case $D=0$. However, the scenario will be different for the case ${\bf(A1)}$. We have proved that there exists a critical value $\overline{D}$, such that \eqref{model} will have periodic solution, through Hopf bifurcation for $D>\overline{D}$. Recalling the results in \cite{chuncheng}, we already know that the memorized diffusion term alone can not lead to the existence of periodic solutions. Therefore, such new phenomenon must be induced by the joint effect of memorized diffusion and spatial heterogeneity.

We remark that the method in this context can be used to study \eqref{model} with Dirichlet boundary condition.
We know that \eqref{eigenvalue prombel} has a unique positive principle eigenvalue $\lambda_*$ with Drichlet boundary under assumption $\bf(A1)$ and $\bf(A2)$.
The existence of steady state $u_{\lambda}$ of \eqref{model} in the small neighborhood of $\lambda_*$ can be proved, witch is analogous to the proof in Proposition \ref{stst}. Using the similar method of proving Theorem \ref{conclusion}, we found that $u_{\lambda}$ is always locally asymptotically stable
for relatively small diffusion rate $D$, which is expected from the results of \eqref{proto} in \cite{diffusion}. However, when $D$ exceeds a critical value $\overline{D}$, Hopf bifurcation will occur by increasing $\tau$ under the same condition in Theorem \ref{conclusion}, but the expression of $u_{\lambda}$ is different.
During the analysis, notice that substituting the constant $m$ for $m(x)$ can not change the result, which implies that only the memory delay can  lead to the occurrence of Hopf bifurcation,  generating spatially inhomogeneous periodic solution, see figure \ref{figureD}. This is different of the conclusion with Neumann boundary.

\begin{figure}[htbp]
\centering
	\begin{minipage}{0.45\linewidth}
		\centering\subfigure[$\tau=2$]
{\includegraphics[scale=0.4]{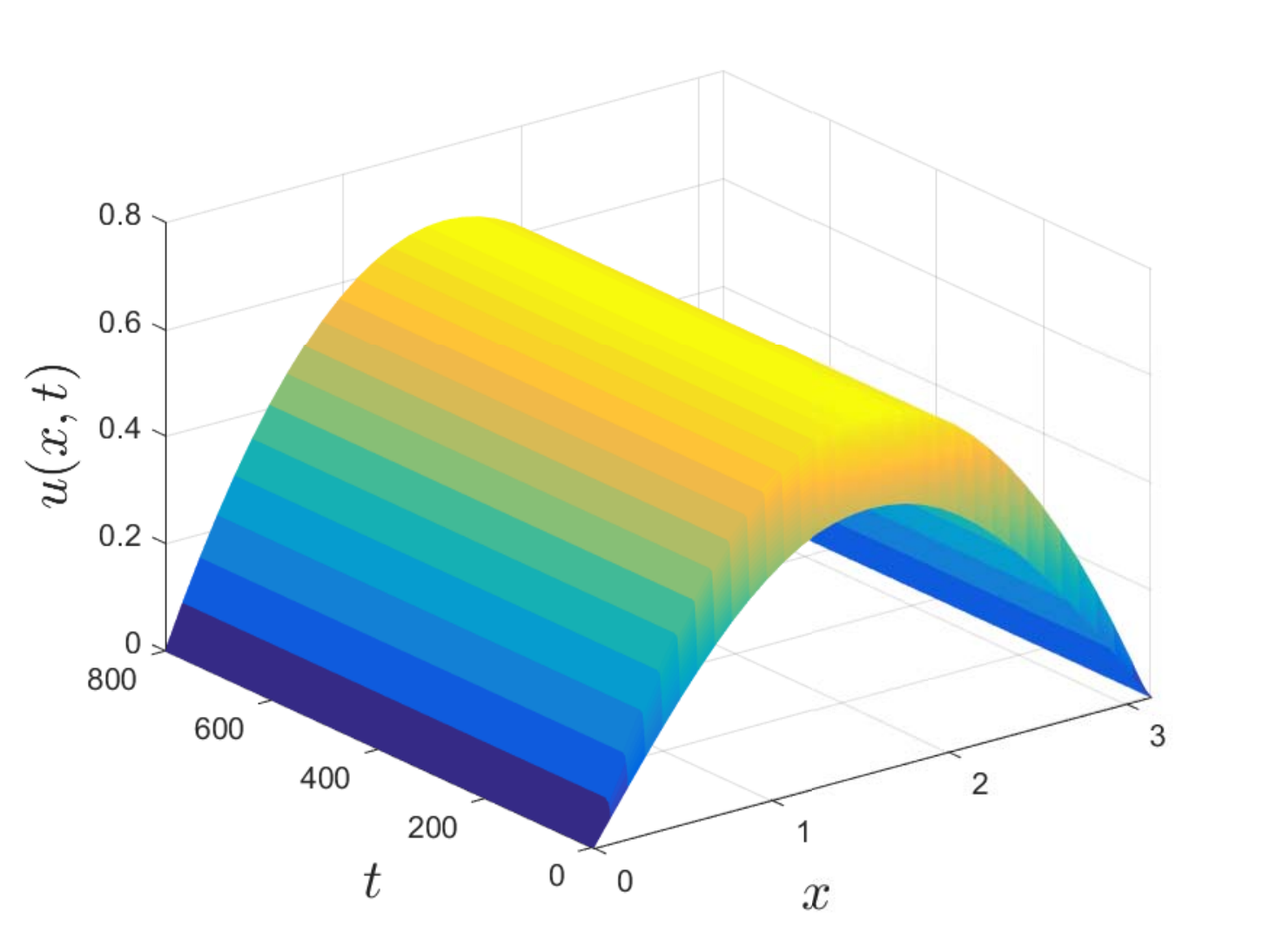}}
			\label{figA}
	\end{minipage}
	\begin{minipage}{0.45\linewidth}
		\centering\subfigure[$\tau=100$]
{\includegraphics[scale=0.4]{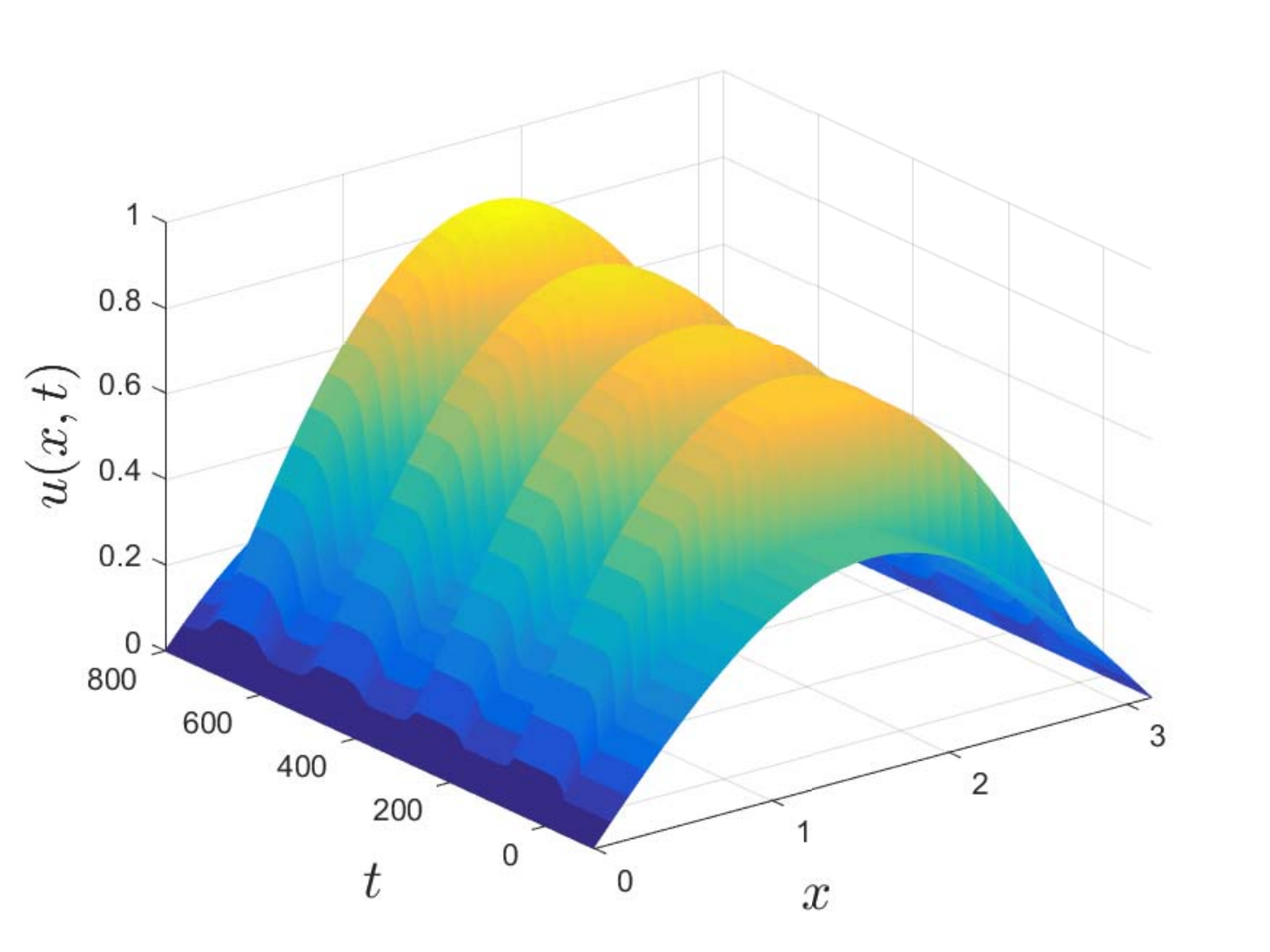}
			\label{figB}}
	\end{minipage}
\caption{Solutions (\ref{model}) with Drichlet boundary condition. $(a)$ The solution will tend to a steady state for $\tau=2$; $(b)$ A periodic solution caused by the increment of $\tau$, through Hopf bifurcation. Here, $m(x)=4$, $\lambda=0.35$ and $D=0.7$, then $\overline{D}=0.5$, $r_1-r_2=\frac{4}{5}>0$ and $r_1+r_2=-\frac{2}{15}<0$.}
\label{figureD}
\end{figure}

\bibliographystyle{plain}

\end{document}